\documentclass[11pt]{amsart}
\usepackage{amssymb,amsmath,amsthm}
\usepackage{amsfonts}
\usepackage{enumerate}
\usepackage{dsfont}
\usepackage{cite}
\usepackage[colorlinks, citecolor=red]{hyperref}
\usepackage{mathrsfs}
\usepackage{epsfig}
\usepackage{lscape}
\usepackage{subfigure}
\usepackage{epstopdf}
\usepackage{caption}
\usepackage{algorithm}
\usepackage{algpseudocode}
\usepackage{booktabs}
\usepackage{multirow}
\usepackage{makecell}
\usepackage{geometry}
\usepackage{paralist}
\usepackage{enumerate}
\usepackage{url}
\usepackage{graphicx,cite}
\usepackage{longtable}
\usepackage{tikz}
\usetikzlibrary{calc}
\usetikzlibrary{arrows.meta}
\usetikzlibrary{decorations.pathreplacing}
\usepackage{lipsum}

\usepackage{tcolorbox}
\tcbuselibrary{theorems}  
\newtcbtheorem[number within=section]{permstrategy}{Permutation Strategy}{%
	colback=white,
	colframe=black,
	colbacktitle=white,
	coltitle=black,
	fonttitle=\bfseries\large,
	boxrule=0.5pt,
	titlerule=0.25pt,
	top=1pt,
	bottom=1pt,
	left=6pt,
	right=6pt
}{strat}

\newcommand{\cA}{{\mathcal A}}
\newcommand{\cB}{{\mathcal B}}
\newcommand{\cC}{{\mathcal C}}
\newcommand{\cD}{{\mathcal D}}

\newcommand{\cE}{{\mathcal E}}

\newcommand{\cG}{{\mathcal G}}
\newcommand{\cH}{{\mathcal H}}
\newcommand{\cI}{{\mathcal I}}

\newcommand{\cK}{{\mathcal K}}

\newcommand{\cP}{{\mathcal P}}

\newcommand{\cR}{{\mathcal R}}

\newcommand{\cT}{{\mathcal T}}
\newcommand{\cU}{{\mathcal U}}

\newcommand{\cX}{{\mathcal X}}

\newcommand{\uf}{\operatorname{unfold}}

\newcommand{\bc}{\operatorname{bcirc}}
\newcommand{\tv}{\operatorname{vec}}

\newcommand{\spn}{\operatorname{span}}

\newtheorem{definition}{Definition}[section]

\newtheorem{prop}[definition]{Proposition}
\newtheorem{theorem}[definition]{Theorem}
\newtheorem{lemma}[definition]{Lemma}

\newtheorem{remark}[definition]{Remark}

\date{}

\begin{document}
\baselineskip 18pt
\title[Linear convergence of TKGK]
{Linear convergence of  Gearhart-Koshy accelerated Kaczmarz methods  for tensor linear systems}

\author{Yijie Wang}
\address{School of Mathematical Sciences, Beihang University, Beijing, 100191, China. }
\email{yj\_wang@buaa.edu.cn}

\author{Yonghan Sun}
\address{School of Mathematical Sciences, Beihang University, Beijing, 100191, China. }
\email{sunyonghan@buaa.edu.cn}

\author{Deren Han}
\address{LMIB of the Ministry of Education, School of Mathematical Sciences, Beihang University, Beijing, 100191, China. }
\email{handr@buaa.edu.cn}

\author{Jiaxin Xie}
\address{LMIB of the Ministry of Education, School of Mathematical Sciences, Beihang University, Beijing, 100191, China. }
\email{xiejx@buaa.edu.cn}

\begin{abstract}
The generalized Gearhart-Koshy acceleration is a recent exact affine search technique designed for the method of cyclic projections onto hyperplanes, i.e., the Kaczmarz method. However, its convergence properties, particularly the linear convergence rate, have not been thoroughly established.
In this paper, we systematically establish the linear convergence of the generalized Gearhart-Koshy accelerated Kaczmarz method for tensor linear systems, proving that it converges linearly to the unique least-norm solution. Our analysis is general and applies to several popular Kaczmarz variants, including incremental, shuffle-once, and random-reshuffling schemes, and demonstrates that this acceleration approach yields a better convergence upper bound compared to the plain Kaczmarz method. We also propose an efficient Gram-Schmidt-based implementation that computes the next iterate in linear time. Building on this implementation, we establish a connection between this acceleration framework and Arnoldi-type Krylov subspace methods, further highlighting its efficiency and potential. Our theoretical results are supported by numerical experiments.
\end{abstract}

\maketitle

\let\thefootnote\relax

\let\thefootnote\relax

\section{Introduction}

The need to process and analyze large-scale datasets is ubiquitous in signal processing and machine learning. While high-dimensional problems are often reformulated as matrix linear systems and addressed using traditional matrix-based techniques, tensor representations offer unique advantages, as demonstrated in \cite{kilmer2013third, soltani2016tensor, newman2020nonnegative,zhou2017tensor}. For instance, although a set of grayscale images can be arranged as column vectors in a matrix, representing each image as a frontal slice of a third-order tensor preserves the intrinsic multilinear structure of the data and mitigates structural distortions. In this paper, we focus on solving large-scale linear systems for third-order tensors under the $t$-product \cite{kilmer2013third}. 

In particular, we consider the tensor linear system
\begin{equation}\label{tensor_system}
	\cA \ast \cX = \cB, \ \ \cA \in \mathbb{R}^{m \times \ell \times n}, \ \ \cB \in \mathbb{R}^{m \times p \times n},
\end{equation}
where $\ast$ denotes the $t$-product introduced by Kilmer and Martin \cite{kilmer2011factorization}; see Section \ref{section-t-product}. 
Iterative methods for solving such tensor linear systems have attracted increasing attention in recent years, as the $t$-product serves as a \emph{natural generalization} of the matrix-vector product, enabling classical linear algebra concepts to be extended to the tensor domain. Within this algebraic framework, the tensor randomized Kaczmarz (TRK) method was first introduced by Ma and Molitor \cite{ma2022randomized} as a generalization of the classical randomized Kaczmarz (RK) algorithm \cite{Kac37, strohmer2009randomized}. 

Recently, a novel acceleration scheme for the method of cyclic projections (the deterministic Kaczmarz method) was proposed, drawing on the Gearhart-Koshy affine search strategy \cite{gearhart1989acceleration, tam2021gearhart, rieger2023generalized}. This method computes the best Euclidean-norm approximation to the solution within specific affine subspaces \cite{rieger2023generalized}  and offers a promising yet underexplored direction. Nevertheless, the existing convergence results \cite[Theorem 10]{rieger2023generalized} only guarantee convergence to some solution, without establishing any convergence rate. In this paper, we provide a comprehensive convergence theory and establish linear convergence rates for this accelerated approach.

\subsection{Our contribution}
The main contributions of this work are as follows.
 
\begin{itemize}
    \item [1.] 
	We extend several popular Kaczmarz variants, including incremental, shuffle-once, and random-reshuffling schemes, to solve the tensor linear system \eqref{tensor_system}, and then incorporate the generalized Gearhart-Koshy acceleration technique into our tensor Kaczmarz (TK) framework. We prove that the resulting TKGK method converges linearly to the unique least-norm solution for any type of tensor \(\cA\), whether under-determined, over-determined, full-rank, or rank-deficient. Furthermore, we demonstrate that its convergence upper bound decreases monotonically as more previous iterates are utilized, and that it can be strictly smaller than that of the plain TK method; see Theorem \ref{TKGK_convergence_thm} and Remarks \ref{remark-0405} and \ref{rmk_more_history_tighter_bound}.
    \item [2.] 
	To determine the next iterate, a linear system must be solved after every epoch of the TKGK method. By exploiting the specific structure of the problem, we propose a truncated Gram-Schmidt process to deal with this linear system. This novel implementation can not only achieve linear-time complexity, but also overcome the limitation of the tridiagonal-based approach \cite[Section 6]{rieger2023generalized}, which requires tensor extraction and concatenation operations; see Section \ref{section4-1}.
	\item [3.] 	Leveraging the proposed Gram-Schmidt-based implementation, we can further characterize the tensor Hessenberg structure and Arnoldi decomposition of TKGK updates, provided that all previous iterates are used and incremental or shuffle-once strategies are employed. We then show that the TKGK method recovers an Arnoldi-type Krylov subspace method, offering insight into the connection between the Gearhart-Koshy acceleration scheme and Krylov subspace techniques; see Section \ref{subsec_arnoldi}. This connection also helps justify the potential and practical appeal of the Gearhart-Koshy acceleration framework. The conclusions are validated by the results of our numerical experiments. 
\end{itemize}


\subsection{Related work}


The Kaczmarz method, introduced in 1937 \cite{karczmarz1937angenaherte}, is a classic and efficient row-action iterative solver for the linear system $Ax=b$, where $A\in\mathbb{R}^{m\times n}$ and $b\in\mathbb{R}^m$. In computed tomography, it is also known as the algebraic reconstruction technique (ART) \cite{herman1993algebraic,gordon1970algebraic}. The method has applications ranging from tomography to digital signal processing. A significant advance is the RK method by Strohmer and Vershynin \cite{strohmer2009randomized}, which, under row-norm-proportional sampling, achieves linear convergence in expectation for consistent systems. Since then, there is a large amount of work on the development of the Kaczmarz-type methods including block Kaczmarz methods \cite{necoara2019faster,needell2014paved,gower2015randomized,xie2025randomized}, accelerated RK methods \cite{liu2016accelerated,han2025pseudoinverse,loizou2020momentum,zeng2024adaptive}, greedy RK methods \cite{bai2018greedy,gower2021adaptive,su2026convergence,su2024greedy}, randomized Douglas-Rachford methods \cite{Han2022-xh}, etc. Recently, RK-type algorithms have also been extended to tensor linear systems within the $t$-product framework, including the TRK method \cite{ma2022randomized}, tensor sketch-and-project method \cite{ma2022randomized,bao2022randomized,tang2023sketch}, regularized Kaczmarz method for tensor recovery \cite{chen2021regularized}, structured factorization-based methods \cite{castillo2024randomized,castillo2025block}, frontal-slice-based method \cite{luo2024frontal}, and accelerated tensor RK methods \cite{liao2024accelerated}.

In fact, the RK method is a special case of the stochastic gradient descent (SGD) method for convex optimization when applied to quadratic objective functions \cite{robbins1951stochastic,Han2022-xh,needell2014stochasticMP,ma2017stochastic,zeng2023randomized}. 
Consequently, many developments in SGD can be adapted to RK. For instance, the random reshuffling (RR) sampling scheme \cite{ahn2020sgd, mishchenko2020random, ying2018stochastic, nguyen2021unified}, which uses sampling without replacement within each epoch, has been introduced into Kaczmarz-type methods \cite{han2025simple}. Unlike standard SGD with independent sampling, RR introduces statistical dependence and eliminates unbiased gradient estimates, making its analysis more challenging.
Nonetheless, RR has been empirically shown to outperform standard SGD in many applications, due in part to its implementation simplicity and the utilization of all samples within each epoch. In this paper, we further consider two additional sampling strategies commonly used in SGD, namely the shuffle-once (SO) and incremental strategies (IS) \cite{mishchenko2020random, safran2020good}. Specifically, our Kaczmarz method can be implemented with IS, SO, and RR sampling schemes, and our convergence analysis provides a unified theoretical guarantee that applies to all three strategies; see Theorem \ref{trrk_convergence_thm}.

The Gearhart-Koshy acceleration scheme was originally proposed as an exact search strategy for alternating projections onto linear subspaces \cite{gearhart1989acceleration}, and was later extended to affine subspaces \cite{tam2021gearhart}. Since the Kaczmarz algorithm is a projection method, it is natural to apply Gearhart-Koshy acceleration to Kaczmarz iterations. More recently, the generalized Gearhart-Koshy acceleration has been incorporated into cyclic (or incremental) Kaczmarz methods \cite{rieger2023generalized}. We note that the main distinction between the generalized and original Gearhart-Koshy acceleration schemes is that the generalized version leverages information from multiple previous iterates, whereas the original uses only a single previous iterate. Subsequently, it was shown in  \cite{hegland2023generalized} that the cyclic (or incremental) block Kaczmarz method with generalized Gearhart-Koshy acceleration can admit a Krylov subspace interpretation when all previous iterates are used. Moreover, if the projection of the coefficient matrix onto the range of $A^\top$ is symmetric positive definite, then such Krylov subspace method converges linearly \cite[Theorem 11]{hegland2023generalized}. 
In this paper, we extend this line of research by applying generalized Gearhart-Koshy acceleration to Kaczmarz methods with IS, SO, and RR sampling strategies. Moreover, our convergence analysis does not require assumptions on the coefficient matrix or impose constraints on the number of previous iterates used, which refines and generalizes existing theoretical results on Gearhart-Koshy acceleration.

Finally, we note that several efficient implementations of Gearhart-Koshy acceleration have been investigated. A tridiagonal-based approach was proposed in \cite[Section 6]{rieger2023generalized}, however, its repeated extraction and concatenation operations introduce considerable memory movement and communication overhead, which may limit practical efficiency. Our Gram-Schmidt-based implementation can avoid these operations and also maintain linear computational cost; see Section~\ref{section4-1}. Additionally, when all previous iterates are used, the authors in \cite[Section 5]{hegland2023generalized} presented an efficient Gram-Schmidt-based scheme for constructing an orthogonal basis for the associated Krylov subspace. More recently, Sun et al.~\cite{sun2025connecting} established a connection between randomized iterative methods and Krylov subspace methods via Gram-Schmidt orthogonalization, further identifying it as a Lanczos-type method when all previous iterates are used and the sketching matrix is deterministic. The present work differs from these approaches in that, when all previous iterates are utilized, our orthogonalization procedure corresponds to an Arnoldi process. In addition, in contrast to \cite{hegland2023generalized}, our approach remains valid for any choice of the truncation parameter, rather than requiring all previous iterates to be used.

\subsection{Organization} 
The remainder of the paper is organized as follows.
After introducing some notations and the TK method in Section~\ref{Sec_Preliminaries}, 
we propose the TKGK method and establish its linear convergence in Section~\ref{Sec_TKGK}.
Section~\ref{Sec_efficient_GS_implementation} develops an efficient Gram-Schmidt-based implementation of TKGK 
and analyzes its Arnoldi-type Krylov subspace properties.
Numerical experiments are presented in Section~\ref{Sec_numer_exp} to verify our theoretical results, and we conclude the paper in Section~\ref{Sec_conclu_rmk}.
Proofs of the results are provided in the Appendix.

\section{Preliminaries}\label{Sec_Preliminaries}
\subsection{Basic notation}

We use  uppercase letters such as $A$ for matrices and calligraphic letters such as $\mathcal{A}$ for tensors. For an integer $m\geq 1$, let $[m]:=\{1,2,\ldots,m\}$. For any matrix $A\in\mathbb{R}^{m\times n}$, we use $A^{\top}$, $A^{\dagger}$, $\|A\|_{F}$, and $\|A\|_{2}$ to denote the transpose, the Moore-Penrose pseudoinverse, the Frobenius norm, and the spectral norm of $A$, respectively. 

\subsection{Tensor basics}
\label{section-t-product}
In this subsection, we briefly review key definitions and fundamental concepts in tensor algebra. We adopt the notational conventions from \cite{kilmer2011factorization, kilmer2013third, miao2020generalized}.

For a third-order tensor \(\mathcal{A} \in \mathbb{R}^{m \times \ell \times n}\), we denote its \((i, j, k)\) entry by \(\mathcal{A}_{ijk}\), and its horizontal, lateral, and frontal slices by \(\mathcal{A}_{i::}\), \(\mathcal{A}_{:i:}\), and \(\mathcal{A}_{::i}\), respectively. For notational convenience, the \(i\)th frontal slice \(\mathcal{A}_{::i}\) is denoted by \(A_i\). The block circulant matrix of \(\mathcal{A}\), denoted \(\operatorname{bcirc}(\mathcal{A})\), is defined by
\[
\operatorname{bcirc}(\mathcal{A}) :=
\begin{bmatrix}
	A_1 & A_n & \cdots & A_2 \\
	A_2 & A_1 & \cdots & A_3 \\
	\vdots & \vdots & \ddots & \vdots \\
	A_n & A_{n-1} & \cdots & A_1
\end{bmatrix} \in \mathbb{R}^{mn \times \ell n}.
\]
The \(\operatorname{unfold}(\cdot)\) and \(\operatorname{fold}(\cdot)\) operators are defined as
\[
\operatorname{unfold}(\mathcal{A}) :=
\begin{bmatrix}
	A_1 \\
	A_2 \\
	\vdots \\
	A_n
\end{bmatrix} \in \mathbb{R}^{mn \times \ell}, \quad
\operatorname{fold}\left(
\begin{bmatrix}
	A_1 \\
	A_2 \\
	\vdots \\
	A_n
\end{bmatrix}
\right) := \mathcal{A}.
\]
To be precise about reshaping, we define the tube-wise vectorization operator \(\tv(\cdot)\) as follows. For \(\mathcal{A} \in \mathbb{R}^{m \times \ell \times n}\),
\[
\tv(\mathcal{A}) :=
\begin{pmatrix}
	\tv(A_1) \\
	\vdots \\
	\tv(A_n)
\end{pmatrix},
\]
where \(\tv(A_i)\) denotes the standard column-wise vectorization of matrix \(A_i\).

Given tensors \(\mathcal{A} \in \mathbb{R}^{m \times \ell \times n}\) and \(\mathcal{B} \in \mathbb{R}^{\ell \times p \times n}\), the $t$-product \(\mathcal{A} * \mathcal{B}\) is defined as
\[
\mathcal{A} * \mathcal{B} := \operatorname{fold}(\operatorname{bcirc}(\mathcal{A}) \cdot \operatorname{unfold}(\mathcal{B})).
\]
The pseudoinverse of \(\mathcal{A} \in \mathbb{R}^{m \times \ell \times n}\), denoted by \(\mathcal{A}^{\dagger}\), is the tensor satisfying
\[
\operatorname{bcirc}(\mathcal{A}^{\dagger}) = \operatorname{bcirc}(\mathcal{A})^{\dagger}.
\]
The transpose of \(\mathcal{A}\), denoted \(\mathcal{A}^{\top}\), is the \(\ell \times m \times n\) tensor obtained by transposing each frontal slice and reversing the order of transposed slices \(2\) through \(n\). It satisfies
\[
\operatorname{bcirc}(\mathcal{A}^{\top}) = \operatorname{bcirc}(\mathcal{A})^{\top}.
\]
The identity tensor \(\mathcal{I} \in \mathbb{R}^{m \times m \times n}\) is defined such that its first frontal slice is the \(m \times m\) identity matrix and all remaining frontal slices are zero.
The inner product between two tensors \(\mathcal{A}, \mathcal{B} \in \mathbb{R}^{m \times \ell \times n}\) is defined as
\[
\langle \mathcal{A}, \mathcal{B} \rangle := \sum_{i,j,k} \mathcal{A}_{ijk} \mathcal{B}_{ijk}.
\]
For any $\ \mathcal{B} \in \mathbb{R}^{\ell \times p \times n},\, \mathcal{C} \in \mathbb{R}^{m \times p \times n}$, it satisfies the identity
$
\langle \mathcal{A} * \mathcal{B}, \mathcal{C} \rangle = \langle \mathcal{B}, \mathcal{A}^{\top} * \mathcal{C} \rangle$.
The Frobenius norm of \(\mathcal{A} \in \mathbb{R}^{m \times \ell \times n}\) are defined as
$
\| \mathcal{A} \|_F := \sqrt{\langle \mathcal{A}, \mathcal{A} \rangle} = 
\left( \sum_{i,j,k} \mathcal{A}_{ijk}^2 \right)^{1/2}.
$
The \(K\)-range of \(\mathcal{A}\) is defined as
$
\operatorname{range}_K(\mathcal{A}) := \left\{ \mathcal{A} * \mathcal{Y} \mid \mathcal{Y} \in \mathbb{R}^{\ell \times p \times n} \right\}.
$
It satisfies
\begin{align}\label{range_Adagger_eq_range_Atop}
	\operatorname{range}_K(\mathcal{A}^{\top}) = \operatorname{range}_K(\mathcal{A}^{\dagger}).
\end{align}
Given tensors  $\cX^1,\ldots, \cX^k \in \mathbb{R}^{\ell\times p\times n}$, their affine and linear spans are denoted as
\[
\operatorname{aff}\{ \cX^1,\ldots, \cX^k\} := \left\{\sum_{i=1}^k \alpha_i \cX^i \, \ \left |  \ \sum_{i=1}^k \alpha_i = 1, \alpha_i \in\mathbb{R} \right.\right\}\]
and $ \operatorname{span}\{ \cX^1,\ldots, \cX^k \}$, respectively.

\subsection{The tensor Kacmzarz method}
\label{section-TK}

The tensor Kaczmarz (TK) method is a Kaczmarz-type iterative algorithm designed for solving tensor linear systems under the $t$-product formulation. Consider the following projector:
\[
\mathcal{P}_{i}(\mathcal{X}) := \mathcal{X} - \mathcal{A}_{i::}^{\dagger} \ast \left( \mathcal{A}_{i::} \ast \mathcal{X} - \mathcal{B}_{i::} \right),
\]
which projects any point \(\mathcal{X} \in \mathbb{R}^{\ell \times p \times n}\) onto the affine subspace \(\{\mathcal{X} \mid \mathcal{A}_{i::} \ast \mathcal{X} = \mathcal{B}_{i::}\}\). Leveraging the algebraic properties of the $t$-product, the TK procedure \cite{ma2022randomized}  is outlined in Algorithm~\ref{tensor-k}. In this paper, we consider the  permutation strategies \(\pi_k = (\pi_{k,1}, \pi_{k,2}, \ldots, \pi_{k,m})\) introduced in Strategy  \ref{strat:perm-strategies}, which have been widely studied in the literature.


\begin{permstrategy}{}{perm-strategies} 

	{\bf Incremental strategy (IS):} A deterministic cyclic order is used in every iteration, with $\pi_k = (1, 2, \ldots, m)$ for all $k \geq 0$ \cite{mishchenko2020random, safran2020good}.
	
	{\bf Shuffle-once (SO):} A single random permutation $(\pi_{0,1}, \pi_{0,2}, \ldots, \pi_{0,m})$ of $[m]$ is generated once at initialization and reused in all subsequent epochs, i.e., $\pi_k = (\pi_{0,1}, \pi_{0,2}, \ldots, \pi_{0,m})$ for all $k \geq 1$ \cite{mishchenko2020random, safran2020good}.

	{\bf Random reshuffling (RR):} At each epoch $k$, a fresh random permutation $\pi_k$ is generated by sampling without replacement from $[m]$ \cite{ahn2020sgd, mishchenko2020random, ying2018stochastic, nguyen2021unified}.
\end{permstrategy}

\begin{algorithm}[htpb]
	\caption{The tensor Kaczmarz (TK) method \label{tensor-k}}
	\begin{algorithmic}
		\Require
		$\cA \in \mathbb{R}^{m\times \ell \times n}$, $\cB \in \mathbb{R}^{m\times p\times n}$, $k=0$, an initial $\cX^0\in \mathbb{R}^{\ell\times p\times n}$, and a permutation strategy $  \Pi\in \{ \text{IS}, \text{RR}, \text{SO}\}$ as defined in Strategy \ref{strat:perm-strategies}.
		\begin{enumerate}
			\item[1:] Set $\cX^{k}_0=\cX^k$ and generate \(\pi_k = (\pi_{k,1}, \ldots, \pi_{k,m})\)  according to strategy $\Pi$.
			\item[2:] {\bf for $i=1,\ldots,m$ do}
			$$
			\cX_{i}^{k}=\mathcal{P}_{\pi_{k,i}}\left(\mathcal{X}_{i-1}^{k}\right)=\mathcal{X}_{i-1}^{k} - \mathcal{A}_{\pi_{k,i}::}^{\dagger} \ast \left( \mathcal{A}_{\pi_{k,i}::} \ast \mathcal{X}_{i-1}^{k} - \mathcal{B}_{\pi_{k,i}::} \right). 
			$$
			{\bf end for}
			\item[3:] Set
			$
			\cX^{k+1}=\cX_{m}^k.
			$
			\item[4:] If the stopping rule is satisfied, stop and go to output. Otherwise, set $k=k+1$ and return to Step $1$.
		\end{enumerate}
		
		\Ensure
		The approximate solution.
	\end{algorithmic}
\end{algorithm}

Let  $ \rho_{\pi_{k}} :=\left\|\bc\left(\cT_{\pi_{k}} \ast \cA^{\dagger} \ast \cA\right) \right\|_{2}^2$ with
\begin{equation}
	\label{Tpi}
	\cT_{\pi_{k}}=  \left(\cI - \cA_{\pi_{k, m}::}^{\dagger} \ast \cA_{\pi_{k, m}::}\right) \ast \cdots \ast \left(\cI - \cA_{\pi_{k, 1}::}^{\dagger} \ast \cA_{\pi_{k, 1}::}\right).
\end{equation}
We now state the convergence result of Algorithm \ref{tensor-k}.
\begin{theorem}
	\label{trrk_convergence_thm}
	Suppose that the tensor linear system $\cA \ast \cX = \cB$ is consistent, and $\cX^0 \in \mathbb{R}^{\ell\times p\times n}$ is an initial tensor.
	Let $\cX_{*}^{0} = \cA^{\dagger} \ast \cB + \left(\cI-\cA^{\dagger}\ast \cA\right) \ast \cX^{0}$. 
	Then the iteration sequence $\{ \cX^k \}$ generated by Algorithm \ref{tensor-k} satisfies
	\[\|\cX^{k+1}-\cX_{*}^{0}\|_{F}^2 \leq \rho_{\pi_{k}}  \|\cX^{k}-\cX_{*}^{0}\|_{F}^2, \]
	with $\rho_{\pi_k} < 1$. 
\end{theorem}


\section{Tensor Kaczmarz with Generalized Gearhart-Koshy acceleration}\label{Sec_TKGK}

In this section, we investigate the generalized Gearhart-Koshy acceleration technique applied to the TK method outlined in Algorithm \ref{tensor-k}, building upon the recent work of Rieger \cite{rieger2023generalized}.
For convenience, for any permutation \(\pi_k = (\pi_{k,1}, \pi_{k,2}, \ldots, \pi_{k,m})\), we define the composite projection operator as
\begin{equation}\label{full-projection-oper}
	\mathcal{P}_{\pi_k}(\mathcal{X}) := \mathcal{P}_{\pi_{k,m}} \circ \cdots \circ \mathcal{P}_{\pi_{k,1}}(\mathcal{X}).
\end{equation}
Using this notation,  Algorithm~\ref{tensor-k} can be concisely expressed as
\[
\mathcal{X}^{k+1} = \mathcal{P}_{\pi_k}(\mathcal{X}^k).
\]

Let \(\tau\) be a fixed positive integer.
At the \(k\)-th iteration, the next iterate \(\cX^{k+1}\) is obtained by performing an exact line search over an affine subspace spanned by the most recent \(\tau\) iterates and the TK update \(\mathcal{P}_{\pi_k}(\cX^k)\). Specifically, we compute
\begin{equation}
	\label{RIM_subspace}
	\begin{aligned}
		\cX^{k+1} = \arg\min & \ \| \cX - \cX^0_{*} \|_F^2 \\
		\text{subject to} \quad & \cX \in \Pi_{k} := \operatorname{aff} \left\{ \cX^{j_{k,\tau}}, \ldots, \cX^k, \mathcal{P}_{\pi_k}(\cX^k) \right\},
	\end{aligned}
\end{equation}
where \(\cX^{0}_{*}= \cA^\dagger \ast \cB + (\cI - \cA^\dagger \ast \cA) \ast \cX^0\)  is the orthogonal projection of the initial point \(\cX^0\) onto the solution set \(\{\cX \in \mathbb{R}^{\ell \times p \times n} : \cA \ast \cX = \cB\}\), and \( j_{k,\tau} := \max\{k - \tau + 1, 0\} \) denotes the index of the earliest iterate used in the current affine subspace. A geometric illustration of the generalized Gearhart-Koshy acceleration is provided in Figure~\ref{GGK-figure}.

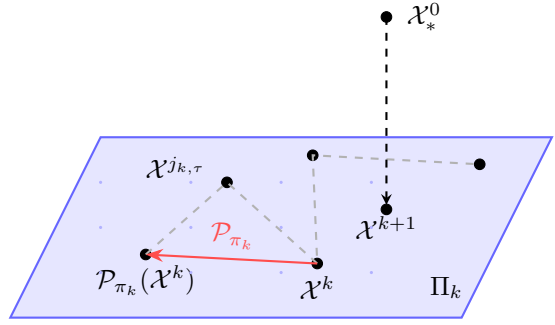
\begin{figure}[hptb]
	\centering
	\begin{tikzpicture}[>=Stealth]		
		\fill[blue!10, opacity=0.4] (0,0)--(6,0)--(7.2,2.4)--(1.2,2.4)--cycle;
		\draw[blue!60, thick] (0,0)--(6,0)--(7.2,2.4)--(1.2,2.4)--cycle;
		\filldraw[black] (1.8,0.84) circle [radius=2pt];      
		\filldraw[black] (2.88,1.8) circle [radius=2pt];
		\filldraw[black] (4.08,0.72) circle [radius=2pt];
		\filldraw[black] (6.24,2.04) circle [radius=2pt];
		\filldraw[black] (4.02,2.16) circle [radius=2pt];
		\filldraw[black] (5,4.0) circle [radius=2pt];
		\filldraw[black] (5,1.44) circle [radius=2pt];      
		\draw[dashed, gray!60, line width=0.8pt] (1.8,0.84) -- (2.88,1.8) -- (4.08,0.72) -- (4.02,2.16) -- (6.24,2.04);
		\draw (1.8,0.5) node[font=\small] {$\mathcal{P}_{\pi_k}(\cX^k)$};
		\draw (4.08,0.4) node[font=\small] {$\cX^k$};
		\draw (2.2,2.0) node[font=\small] {$\cX^{j_{k,\tau}}$};
		\draw (5.5,4.0) node[font=\small] {$\cX_{*}^{0}$};
		\draw (5.0,1.2) node[font=\small] {$\cX^{k+1}$};
		\draw (5.8,0.4) node[font=\small] {$\Pi_{k}$};
		\draw [dashed,-stealth, line width=0.8pt] (5,4.0) -- (5,1.5);
		\draw [->, red!70, thick, line width=0.8pt] ($(4.08,0.72)$) -- ($(1.8,0.84)$)
		node[midway, above, sloped, font=\small] {$\mathcal{P}_{\pi_k}$};
		\foreach \x in {1.2,2.4,3.6,4.8} {
			\foreach \y in {0.6,1.2,1.8} {
				\fill[blue!30, opacity=0.3] (\x,\y) circle [radius=0.6pt];
			}
		}
	\end{tikzpicture}
	\caption{A geometric interpretation of the generalized Gearhart-Koshy acceleration. The next iterate \(\cX^{k+1}\) is determined  by finding the point in the affine subspace \(\Pi_{k} = \operatorname{aff} \left\{ \cX^{j_{k,\tau}}, \ldots, \cX^k, \mathcal{P}_{\pi_k}(\cX^k) \right\}\) that is closest to \(\cX_{*}^{0}\).
	}    
	\label{GGK-figure}
\end{figure}

To explicitly formulate the optimization problem \eqref{RIM_subspace}, we define the matrix \(M_k\) as
\begin{equation} \label{dfmk}
	M_k := \begin{bmatrix}\vec{x}^{j_{k,\tau}} - \vec{x}^k, \ldots,\vec{x}^{k-1} - \vec{x}^k, \vec{d}^k\end{bmatrix} \in \mathbb{R}^{\ell pn \times (k-j_{k,\tau}+1)},	
\end{equation}
where \(\vec{x}^i := \tv(\cX^i)\) and  \(\vec{d}^k := \tv(\mathcal{P}_{\pi_k}(\cX^k) - \cX^k)\) denote the tube vectorization of \(\cX^i\) and $\mathcal{P}_{\pi_k}(\cX^k) - \cX^k$, respectively. 
Then, the minimization problem in \eqref{RIM_subspace} reduces to solving a least-squares problem of the form
\[
s^k = \arg\min_{s \in \mathbb{R}^{k - j_{k,\tau} + 1}} \left\| \vec{x}^k - \vec{x}^{0}_{*} + M_k s \right\|_2^2,
\]
whose optimality condition leads to the normal equation
\begin{equation} \label{optcond2}
	M_k^\top M_k s^k = -M_k^\top (\vec{x}^k - \vec{x}^{0}_{*}).
\end{equation}


We now analyze the right-hand side of \eqref{optcond2}. Since \(\cX^k\) is the orthogonal projection of \(\cX^0_{*}\) onto the affine subspace \(\Pi_{k-1}\), as illustrated in Figure~\ref{GGK-figure}, it follows that
\[
-\langle\vec{x}^{j_{k,\tau}+i-1}- \vec{x}^{k} , \vec{x}^{k}-\vec{x}_{*}^{0}   \rangle=\langle \cX^{j_{k,\tau}+i-1}- \cX^{k} , \cX_{*}^{0} - \cX^{k} \rangle =0,\ \forall \ 1\leq i\leq k-j_{k,\tau}.
\]
 Define
\begin{align}\label{def_gammak}
	\gamma_k:=-\left\langle \vec{d}^k , \vec{x}^{k}-\vec{x}_{*}^0 \right\rangle=\left\langle P_{\pi_{k}}\left(\cX^{k}\right)-\cX^{k}, \cX_{*}^{0} - \cX^{k} \right\rangle .
\end{align}
Then, the normal equation \eqref{optcond2} can be rewritten more concisely as
\begin{equation} \label{eq-s_k}
	M_k^\top M_k s^k = \gamma_k\, e^{k-j_{k,\tau}+1}_{k-j_{k,\tau}+1},
\end{equation}
where \(e^{k-j_{k,\tau}+1}_{k-j_{k,\tau}+1} \in \mathbb{R}^{k-j_{k,\tau}+1}\) is the standard basis vector with a $1$ at the last coordinate.


Now, the generalized Gearhart-Koshy acceleration reduces to solving the linear system \eqref{eq-s_k}. However, this process presents two main challenges. First, as evident from the definition of \(\gamma_k\), it depends on the true solution \(\cX_{*}^{0}\), which is typically unknown in practice. As a result, \(\gamma_k\) appears to be intractable. Second, it is unclear whether the linear system \eqref{eq-s_k} admits a unique solution, since this depends on the rank of the matrix \(M_k\).

The following results address both issues under the assumption that \(\cX^k\) is not a solution to the tensor linear system \(\cA \ast \cX = \cB\). For convenience,
we define $r_{\pi_{k}}:\mathbb{R}^{\ell\times p\times n} \rightarrow \mathbb{R}^{mn\ell \times p}$ as
\begin{align}\label{def_r_k}
	r_{\pi_{k}}\left(\cX\right):=\begin{bmatrix} 
		\uf\left(\cA_{\pi_{k, 1}::}^\dagger \ast \left(\cA_{\pi_{k, 1}::} \ast \cX - \cB_{\pi_{k, 1}::}\right)\right)\\
		\uf\left(\cA_{\pi_{k, 2}::}^\dagger \ast \left(\cA_{\pi_{k, 2}::} \ast P_{\pi_{k, 1}}\left(\cX\right) - \cB_{\pi_{k, 2}::}\right)\right)\\
		\vdots\\
		\uf\left(\cA_{\pi_{k, m}::}^\dagger \ast \left(\cA_{\pi_{k, m}::} \ast P_{\pi_{k, m-1}} \circ \ldots \circ P_{\pi_{k, 1}}\left(\cX\right) - \cB_{\pi_{k, m}::}\right)\right)
	\end{bmatrix}.
\end{align}

\begin{lemma} \label{equiv}
	A tensor \(\cX^k\) satisfies \(\cA \ast \cX^k = \cB\) if and only if \(r_{\pi_k}(\cX^k) = 0\), if and only if \(\mathcal{P}_{\pi_k}(\cX^k) = \cX^k\).
	Moreover, it follows that $\cA\ast\cX^k\neq\cB$ implies $\cA\ast\cX^i\neq\cB, i=0,1,\ldots,k$ for any $k\geq0$.
\end{lemma}

\begin{lemma} \label{ensure-comput}
	For any \(k \ge 0\), the quantity \(\gamma_k\) in \eqref{eq-s_k} can be computed as
	\begin{align}\label{gammak}
		\gamma_k=\frac{1}{2}\left\|r_{\pi_{k}}\left(\cX^k\right)\right\|_{F}^2+\frac{1}{2}\left\|\cP_{\pi_{k}}\left(\cX^k\right) - \cX^k\right\|_{F}^2,
	\end{align}
	where $\mathcal{P}_{\pi_k}$ and $r_{\pi_{k}}$  are defined in \eqref{full-projection-oper} and \eqref{def_r_k}, respectively.
	Furthermore, if \(\cA \ast \cX^k \ne \cB\), then  \(\gamma_k > 0\), and the matrix \(M_k\) has full column rank. Consequently, the linear system \eqref{eq-s_k} admits a unique solution.
\end{lemma}

We present our TK with generalized Gearhart-Koshy acceleration in Algorithm~\ref{TKGK}. 
In particular, when the truncation parameter $\tau = 1$, Algorithm~\ref{TKGK} reduces to the original Gearhart-Koshy acceleration \cite{tam2021gearhart}.

\begin{algorithm}[htpb]
	\caption{Tensor Kaczmarz with generalized Gearhart-Koshy acceleration (TKGK) \label{TKGK} }
	\begin{algorithmic}
		\Require
		$\cA \in \mathbb{R}^{m\times \ell \times n}$, $\cB \in \mathbb{R}^{m\times p\times n}$, positive integer $\tau \geq 1$, $k=0$, an initial $\cX^0\in \mathbb{R}^{\ell\times p\times n}$, and a permutation strategy $\Pi\in \{ \text{IS}, \text{RR}, \text{SO}\}$ as defined in Strategy \ref{strat:perm-strategies}.
		\begin{enumerate}
			\item[1:] Generate \(\pi_k = (\pi_{k,1}, \ldots, \pi_{k,m})\) according to strategy $\Pi$. 
			\item[2:] Compute $\cP_{\pi_k}(\cX^k)$ and  $\delta_{k} = \left\|\cP_{\pi_k}\left(\cX^k\right) - \cX^k\right\|_{F}^2$.
			\item[3:] If $\delta_{k} = 0$, stop and go to output.
			\item[4:] Compute $\rho_{k} = \left\| r_{\pi_{k}}\left(\cX^k\right)\right\|_{F}^2$ and $\gamma_{k} = \frac{1}{2} \left(\rho_{k}+\delta_{k}\right)$.
			\item[5:] Update  $j_{k,\tau}=\max\{k-\tau+1,0\}$ and compute $M_{k}^{\top}M_{k}$.
			\item[6:] Find \( s_k \) as the solution to \eqref{eq-s_k}, i.e.,
			$$	M_k^\top M_ks^k=\gamma_ke_{k-j_{k,\tau}+1}^{k-j_{k,\tau}+1}.$$   
			\item[7:] Update $$\cX^{k+1}=\cX^k+\sum_{i=1}^{k-j_{k,\tau}}s_{i}^{k}\left(\cX^{j_{k,\tau}+i-1}-\cX^{k}\right)+s_{k-j_{k,\tau}+1}^{k}\left(\cP_{\pi_{k}}\left(\cX^{k}\right)-\cX^{k}\right).$$ 
			\item[8:] If the stopping rule is satisfied, stop and go to output. Otherwise, set $k=k+1$ and return to Step $1$.
		\end{enumerate}
		
		\Ensure
		The approximate solution.
	\end{algorithmic}
\end{algorithm}

\begin{remark}
We note that $\delta_{k} = 0$, i.e., $\cP_{\pi_k}(\cX^k) - \cX^k = 0$, can serve as a valid stopping criterion that guarantees the algorithm terminates exactly at a solution. Indeed, by Lemma~\ref{equiv}, if $\cP_{\pi_k}(\cX^k) - \cX^k = 0$, then $\cA \ast \cX^k = \cB$.
\end{remark}


\begin{remark}
	\label{complexity_TKGK}
	Let us consider the computational cost of Step 6 of Algorithm~\ref{TKGK}, which requires solving a linear system. 
	Indeed, without exploiting any potential iterative structure of \(M_k\), this step involves two main procedures  forming the Gram matrix \(M_k^\top M_k\) and solving the resulting system.
	Note that while \(M_k\) is defined via vectorization in \eqref{dfmk}, the Gram matrix \(M_k^\top M_k\) can be computed more directly using tensor inner products
	\begin{align}\label{Gram_matrix}
		M_k^\top M_k =
		\begin{bmatrix}
			\langle \cX^{j_{k,\tau}}-\cX^k , \cX^{j_{k,\tau}}-\cX^k \rangle & \cdots &
			\langle \cX^{j_{k,\tau}}-\cX^k , P_{\pi_k}(\cX^k)-\cX^k \rangle\\
			\vdots & & \vdots\\
			\langle P_{\pi_k}(\cX^k)-\cX^k , \cX^{j_{k,\tau}}-\cX^k \rangle & \cdots &
			\langle P_{\pi_k}(\cX^k)-\cX^k , P_{\pi_k}(\cX^k)-\cX^k \rangle
		\end{bmatrix}.
	\end{align}
	Constructing this matrix requires  \(\mathcal{O}(\tau^2 \ell p n)\) floating-point operations. Subsequently, solving the normal equations, for example via Cholesky factorization, needs  \(\mathcal{O}(\tau^3)\) cost.  Hence, given that the truncation parameter \(\tau\) is typically much smaller than \(\ell p n\), the overall computational cost per iteration is dominated by \(\mathcal{O}(\tau^2 \ell p n)\).
	
	In fact, by exploiting the tridiagonal structure of the inverse of the $M_k^{\top}M_k$ submatrix and using the Schur complement technique, a tridiagonal-based implementation was proposed in \cite[Section 6]{rieger2023generalized}. 
	This approach reduces the computational cost to \(\mathcal{O}(\tau \ell p n)\), which is linear in the problem dimension. In this paper, we present an efficient implementation strategy based on Gram--Schmidt orthogonalization. Our method not only achieves the same \(\mathcal{O}(\tau \ell p n)\) computational cost but also avoids the tensor extraction or concatenation operations required by the tridiagonal-based implementation approach in \cite[Section 6]{rieger2023generalized}; see Section~\ref{section4-1}.
\end{remark}

\subsection{Convergence analysis}
Let us introduce some auxiliary variables.
Define
\begin{align}\label{def_vk}
	V_0 := \emptyset, \quad 
	V_k := \begin{bmatrix} \vec{x}^{j_{k,\tau}} - \vec{x}^k, \ldots, \vec{x}^{k-1} - \vec{x}^k \end{bmatrix} \in \mathbb{R}^{\ell p n \times (k - j_{k,\tau})}, \quad \text{for } k \ge 1,
\end{align}
where \(\vec{x}^i = \tv(\cX^i)\). We next define two scalar quantities that will be used in the convergence analysis
\begin{equation} \label{rate-betak}
	\beta_k := \left(1 - \frac{\left\| V_k V_k^{\dagger} \vec{d}^k \right\|_2^2}{\left\| \vec{d}^k \right\|_2^2} \right)^{-1}, \quad
	\zeta_k := \frac{ \left\langle \cX^k - \cX_{*}^{0},\, \cX^k - \cP_{\pi_k}(\cX^k) \right\rangle }{ \left\| \cX^k - \cX_{*}^{0} \right\|_F \left\| \cP_{\pi_k}(\cX^k) - \cX^k \right\|_F },
\end{equation}
where \(\vec{d}^k = \tv(\mathcal{P}_{\pi_k}(\cX^k) - \cX^k)\). 
By Lemma~\ref{equiv}, if \(\cX^k\) is not a solution to the tensor linear system \(\cA \ast \cX = \cB\), 
then \(\vec{d}^k \ne 0\), \(\cP_{\pi_k}(\cX^k) \ne \cX^k\), and \(\cX^k \ne \cX_{*}^{0}\). 
As a result, both \(\beta_k\) and \(\zeta_k\) are well-defined.

We now state the convergence result of the TKGK method.

\begin{theorem}
	\label{TKGK_convergence_thm}
	Suppose that the tensor linear system $\cA \ast \cX = \cB$ is consistent, and that $\cX^0 \in \mathbb{R}^{\ell \times p\times n}$ is an arbitrary initial tensor.
	Define $\cX_{*}^{0} = \cA^{\dagger} \ast \cB + \left(\cI-\cA^{\dagger}\ast \cA\right) \ast \cX^{0}$, and let $\{ \cX^k \}_{k\geq0}$ be the sequence generated by Algorithm \ref{TKGK}. 
	If $\cA\ast \cX^k=\cB$, then $\cX^k=\cX_{*}^{0}$. Otherwise,
	\[\left\|\cX^{k+1} - \cX_{*}^{0}\right\|_{F}^{2} \leq \left(1 - \beta_{k}\zeta_{k}^{2}\right) \left\|\cX^{k}-\cX_{*}^{0}\right\|_{F}^2, \]
	where $\beta_{k}$ and $\zeta_{k}$ are defined as \eqref{rate-betak}. 
	The convergence factor is better than that of the TK method in Algorithm \ref{tensor-k}, i.e.,
	\begin{align*}
		1 - \beta_{k} \zeta_{k}^{2} \leq \rho_{\pi_k}.
	\end{align*}
	Moreover, this inequality is strict if $\| r_{\pi_{k}}(\mathcal{X}^{k}) \|_{F}^{2} \neq \| \mathcal{X}^{k} - \mathcal{P}_{\pi_{k}}(\mathcal{X}^{k}) \|_{F}^{2}$. 
\end{theorem}

The following remark shows that, with a proper initial tensor, TKGK converges linearly to the unique least-norm solution $\cA^{\dagger} \ast \cB$.
\begin{remark}\label{remark-0405}
	Let $P_m$ denote the set of all permutations of $[m]$, and define
	\(
	\rho := \max_{\pi\in P_m} \rho_{\pi}.
	\)
	By Theorem~\ref{TKGK_convergence_thm}, the iteration sequence $\{ \cX^k \}_{k\geq0}$ satisfies
	\[
	\left\|\cX^{k} - \cX_{*}^{0}\right\|_{F}^{2} \leq \rho^k \left\|\cX^{0}-\cX_{*}^{0}\right\|_{F}^2.
	\]
	Since $P_m$ is finite and $\rho_{\pi}<1$ for all $\pi\in P_m$, it follows that $\rho<1$. Furthermore, if the initial tensor satisfies $\cX^0 \in \operatorname{range}_K(\cA^\top)$ (e.g., $\cX^0=0$), then $\cX_*^0 = \cA^{\dagger} \ast \cB$. This implies that the iteration sequence $\{\cX^k\}_{k\ge0}$ generated by Algorithm~\ref{TKGK} converges to the  unique least-norm solution $\cA^{\dagger} \ast \cB$. 
\end{remark}

The following remark illustrates that incorporating more previous iterates can improve the convergence upper bound of Algorithm~\ref{TKGK}.
\begin{remark}\label{rmk_more_history_tighter_bound}
	Consider the case where a smaller truncation parameter \(\tilde{\tau} < \tau\) is used. Define
	\[
	\tilde{V}_k := 
	\begin{bmatrix}
		\vec{x}^{\tilde{j}_{k,\tilde{\tau}}} - \vec{x}^k, \ldots, \vec{x}^{k-1} - \vec{x}^k
	\end{bmatrix},
	\quad \text{where} \quad \tilde{j}_{k,\tilde{\tau}} := \max\{k - \tilde{\tau} + 1, 0\}.
	\]
	Since \(\tilde{\tau} < \tau\), it holds that \(\operatorname{Range}(\tilde{V}_k) \subseteq \operatorname{Range}(V_k)\), which implies
	\[
	\frac{\lVert \tilde{V}_k \tilde{V}_k^\dagger \vec{d}^k \rVert_2^2}{\lVert \vec{d}^k \rVert_2^2} 
	\leq 
	\frac{\lVert V_k V_k^\dagger \vec{d}^k \rVert_2^2}{\lVert \vec{d}^k \rVert_2^2}.
	\]
	Accordingly, the modified convergence rate quantity
	\[
	\tilde{\beta}_k := \left(1 - \frac{\lVert \tilde{V}_k \tilde{V}_k^\dagger \vec{d}^k \rVert_2^2}{\lVert \vec{d}^k \rVert_2^2} \right)^{-1}
	\]
	satisfies \(\tilde{\beta}_k \leq \beta_k\). Moreover, from the definition of \(\zeta_k\) in \eqref{rate-betak}, it is independent of the value of \(\tau\) or \(\tilde{\tau}\). Therefore,
	\[
	1 - \tilde{\beta}_k \zeta_k^2 \geq 1 - \beta_k \zeta_k^2.
	\]
	This confirms that using a larger number of previous iterates (i.e., a larger \(\tau\)) can yield a tighter convergence upper bound for Algorithm~\ref{TKGK}.
\end{remark}

\section{An efficient Gram-Schmidt-based implementation}\label{Sec_efficient_GS_implementation}

This section aims to develop a novel Gram-Schmidt-based implementation to efficiently solve the linear system \eqref{eq-s_k}, which must be solved at every iteration of Algorithm~\ref{TKGK} in Step~6. In particular, we show that the next iterate $\cX^{k+1}$ can be obtained in linear time by exploiting the  structure of the problem. Moreover, using this implementation, Algorithm~\ref{G-S-based TKGK} can recover the Arnoldi-type Krylov subspace method \cite[Section~10.5.1]{golub2013matrix}.

For convenience, we  adopt a special decomposition for vector
\[
x = \begin{pmatrix} \overline{x} \\ \underline{x} \end{pmatrix}\in\mathbb{R}^m, \quad \text{where } \overline{x} := (x_1,\ldots,x_{m-1})^\top \in \mathbb{R}^{m-1} \text{ and } \underline{x} := x_m \in \mathbb{R}.
\]
We use $\mathbf{0}_{m}$ to denote the zero vector in $\mathbb{R}^{m}$. Recall that \(\vec{x}^i = \tv(\cX^i)\) and  \(\vec{d}^k = \tv(\mathcal{P}_{\pi_k}(\cX^k) - \cX^k)\) denote the tube vectorization of \(\cX^i\) and $\mathcal{P}_{\pi_k}(\cX^k) - \cX^k$, respectively. From the definition of $V_k$ in \eqref{def_vk}, we have $M_k = [V_k, \vec{d}^k]$. Hence, the linear system \eqref{eq-s_k} can be equivalently rewritten as
\begin{align}\label{normal_equation_block}
	\begin{bmatrix}
		V_k^{\top}V_k & V_k^{\top}\vec{d}^k \\
		(\vec{d}^k)^{\top}V_k & \|\vec{d}^k\|_{2}^{2}
	\end{bmatrix}
	\begin{pmatrix}
		\overline{s^k} \\
		\underline{s^k}
	\end{pmatrix} = 
	\begin{pmatrix}
		\mathbf{0}_{k-j_{k,\tau}} \\
		\gamma_k
	\end{pmatrix}.
\end{align}
From which we obtain 
\begin{equation}\label{oversk_undersk}
	\overline{s^k} = - \left(V_k^{\top}V_k\right)^{-1}V_k^{\top} \vec{d}^k \underline{s^k} \ \ \text{and} \ \ \underline{s^k}
	=
	\frac{\gamma_k}
	{\langle \vec d^k, (I - V_k (V_k^\top V_k)^{-1} V_k^\top)\vec d^k \rangle}. 
\end{equation}
Here, $V_k^{\top}V_k$ is invertible because $M_k$ has full column rank by Lemma~\ref{ensure-comput}, which implies that the submatrix $V_k$ also has full column rank.
Then, from Step $7$ of Algorithm~\ref{TKGK}, we have
\[
\begin{aligned}
	\vec{x}^{k+1} 
	& = \vec{x}^{k} + s_{1}^k\left(\vec{x}^{j_{k,\tau}} - \vec{x}^{k}\right) + \cdots + s_{k-j_{k,\tau}}^k\left(\vec{x}^{k-1} - \vec{x}^{k}\right) + s_{k-j_{k,\tau}+1}^k\vec{d}^k \\
	& = \vec{x}^{k} + V_k\overline{s^k} + \underline{s^k}\vec{d}^k\\
	& = \vec{x}^{k} + \underline{s^k}(I -V_k (V_k^\top V_k)^{-1} V_k^\top)\vec d^k,
\end{aligned}
\]
where the third equality follows from \eqref{oversk_undersk}. 
Note that $V_k V_k^\dagger = V_k (V_k^\top V_k)^{-1} V_k^\top$, we can get
\begin{equation}\label{vec_update}
	\vec x^{k+1} = \vec x^k + \underline{s^k}(I - V_k V_k^\dagger)\vec d^k \ \ \text{and} \ \ \underline{s^k}=\frac{\gamma_k}{\langle\vec{d}^k, (I - V_kV_k^{\dagger})\vec d^k\rangle}.
\end{equation}

Suppose that $U_k = [\vec{u}_{j_{k,\tau}},\ldots,\vec{u}_{k-1}]$ is an orthogonal basis of $\operatorname{Range}(V_k)$, then we have
\[
\left(I - V_k V_k^\dagger\right)\vec d^k =\left( I - U_k(U_k^{\top}U_k)^{-1}U_k^{\top}\right) \vec d^k = \vec d^k
-
\sum_{i=j_{k,\tau}}^{k-1}
\frac{\langle \vec u_i,\vec d^k\rangle}{\|\vec u_i\|_2^2}\vec u_i
\]
and
\begin{align}\label{idempotent_I-V_kV_kdagger}
	\langle\vec{d}^k, (I - V_kV_k^{\dagger})\vec d^k\rangle = \langle (I - V_kV_k^{\dagger})\vec{d}^k, (I - V_kV_k^{\dagger})\vec d^k\rangle = \left\|(I - V_kV_k^{\dagger})\vec d^k\right\|^2_2.
\end{align}
Let $\vec{u}_k := (I - V_kV_k^{\dagger})\vec d^k$. Then \eqref{vec_update} can be rewritten equivalently in vector and tensor forms as
\begin{equation}\label{xie-eq-0323}
	\left\{
	\begin{aligned}
		\vec{u}_k &= \vec d^k - \sum_{i=j_{k,\tau}}^{k-1}\frac{\langle \vec u_i,\vec d^k\rangle}{\|\vec u_i\|_2^2}\vec u_i,\\
		\vec x^{k+1} &= \vec x^k + \frac{\gamma_{k}}{\|\vec{u}_k\|^2_2}\vec{u}_k,
	\end{aligned}
	\right.
	\iff
	\left\{
	\begin{aligned}
		\cU_k &= \cD_k - \sum_{i=j_{k,\tau}}^{k-1} \frac{\langle \cU_i,\cD_k\rangle}{\|\cU_i\|_F^2}\cU_i,\\
		\cX^{k+1} &= \cX^k + \frac{\gamma_{k}}{\|\cU_k\|^2_F} \cU_k,
	\end{aligned}
	\right.
\end{equation}
where $\gamma_k$ is defined in \eqref{gammak}. The following proposition shows that \eqref{xie-eq-0323} is exactly a Gram-Schmidt orthogonalization procedure, with the initial $\cU_0 = \cD_0 = \cP_{\pi_0}(\cX^0) - \cX^0$, which generates an orthogonal basis of 
\(
\spn\{\cX^{j_{k,\tau}}-\cX^k,\ldots,\cX^{k-1}-\cX^k\}.
\)

\begin{prop}\label{basis_update_lemma}
	For any $k \ge 0$, let $\cD_k = \mathcal{P}_{\pi_k}(\cX^k) - \cX^k$ and $j_{k,\tau} = \max\{k - \tau + 1, 0\}$. Let $\{\cX^k\}_{k\ge0}$ and $\{\cU_k\}_{k\ge0}$ be the sequences generated by the tensor formulation in \eqref{xie-eq-0323}. Then $\langle\cU_i, \cU_j\rangle = 0$ for all $i \neq j$ with $j_{k,\tau} \le i, j \le k$, and
	\[
	\spn\{\cU_{j_{k,\tau}}, \ldots, \cU_{k-1}\}
	=
	\spn\{\cX^{j_{k,\tau}} - \cX^k, \ldots, \cX^{k-1} - \cX^k\}.
	\]
\end{prop}

Now, we have already constructed the Gram-Schmidt-based TKGK method  described in Algorithm \ref{G-S-based TKGK}.

\begin{algorithm}[htpb]
	\caption{Efficient Gram-Schmidt-based TKGK \label{G-S-based TKGK}}
	\begin{algorithmic}
		\Require
		$\cA \in \mathbb{R}^{m\times \ell \times n}$, $\cB \in \mathbb{R}^{m\times p\times n}$, positive integer $\tau \geq 1$, $k=0$, an initial $\cX^0\in \mathbb{R}^{\ell\times p\times n}$, and a permutation strategy $\Pi\in \{ \text{IS}, \text{RR}, \text{SO}\}$ as defined in Strategy \ref{strat:perm-strategies}.
		\begin{enumerate}
			\item[1:] Generate \(\pi_0 = (\pi_{0,1}, \ldots, \pi_{0,m})\) according to the strategy $\Pi$. 
			\item[2:] Compute $\cD_0 = \cP_{\pi_0}\left(\cX^0\right) - \cX^0$. 
			\item[3:] If $\cD_0 = 0$, stop and go to output. Otherwise, set $\cU_0 = \cD_0$.
			\item[4:] Compute $\gamma_k = \frac{1}{2}\left(\left\|r_{\pi_{k}}\left(\cX^k\right)\right\|_{F}^{2} + \left\|\cD_k\right\|_{F}^{2}\right)$ 
			and $\lambda_k = \frac{\gamma_k}{\left\|\cU_k\right\|_{F}^2}$.
			\item[5:] Update $\cX^{k+1} = \cX^k + \lambda_k \cU_k$.
			\item[6:] Generate \(\pi_{k+1} = (\pi_{k+1,1}, \ldots, \pi_{k+1,m})\) according  to the strategy $\Pi$. 
			\item[7:] Compute $\cD_{k+1} = \cP_{\pi_{k+1}}\left(\cX^{k+1}\right) - \cX^{k+1}$.
			\item[8:] If $\cD_{k+1} = 0$, stop and go to output. Otherwise, update $j_{k+1,\tau} = \max\{k - \tau + 2, 0\}$ and compute 
			$$	\cU_{k+1} = \cD_{k+1} - \sum_{i=j_{k+1,\tau}}^{k} \frac{\langle\cU_i, \cD_{k+1}\rangle}{\left\|\cU_i\right\|_{F}^{2}}\cU_i. $$
			\item[9:] If the stopping rule is satisfied, stop and go to output. Otherwise, set $k = k + 1$ and return to Step $4$.
		\end{enumerate}
		
		\Ensure
		The approximate solution.
	\end{algorithmic}
\end{algorithm}

The following theorem demonstrates the equivalence between Algorithms \ref{TKGK} and \ref{G-S-based TKGK}.

\begin{theorem}\label{thm_equiv_Algo2_Algo3}
	If Algorithm \ref{TKGK} and Algorithm \ref{G-S-based TKGK} have the identical initial tensor $\cX^0$ and
	permutation sequence $\{\pi_k\}_{k \ge 0}$, 
	then their generated sequences $\{\cX^k\}_{k \ge 0}$ coincide exactly. 
\end{theorem}

Comparing Algorithm~\ref{TKGK} and Algorithm~\ref{G-S-based TKGK}, we see that Step 6 of Algorithm~\ref{TKGK}, which solves the linear system \eqref{eq-s_k}, now becomes a Gram-Schmidt orthogonalization. This procedure involves at most $\tau-1$ inner products $\langle\cU_i, \cD_{k+1}\rangle$, with $\|\cU_i\|_F^2$ computed in advance, each costing $\mathcal{O}(\ell p n)$ operations, resulting in $\mathcal{O}(\tau \ell p n)$ flops per iteration. This is linear in the problem dimension and much cheaper than directly solving \eqref{eq-s_k}, which requires $\mathcal{O}(\tau^2 \ell p n)$ operations; see Remark \ref{complexity_TKGK}.

\subsection{Comparison with the tridiagonal-based implementation}
\label{section4-1}

In this section, we compare the proposed Gram-Schmidt-based implementation with the tridiagonal-based approach proposed in \cite[Section 6]{rieger2023generalized}.  In particular, by exploiting the structure of the matrices $V_k^\top V_k$, it was shown in \cite[Lemma 12 and Theorem 14]{rieger2023generalized} that the inverse of $V_k^\top V_k$ admits a tridiagonal form. The Schur complement can then be used to obtain an explicit solution to the structured linear system \eqref{normal_equation_block}. This approach avoids forming the Gram matrix $M_k^\top M_k$ entirely, and matrix multiplications involving tridiagonal matrices can be performed in linear time. Nevertheless, we note that although both implementations achieve linear-time complexity, their practical procedures can differ substantially.

In fact, the tridiagonal-based approach developed in \cite[Section 6]{rieger2023generalized} can be extended naturally to tensor linear systems, as the $t$-product serves as a natural generalization of the matrix-vector product. In this implementation, the tensor array
$ [\cX^{j_{k,\tau}}, \ldots, \cX^k]
$
should be stored explicitly. To construct
\[
[\cX^{j_{k,\tau}} - \cX^k, \ldots, \cX^{k-1} - \cX^k],
\]
it is necessary to extract $[\cX^{j_{k,\tau}}, \ldots, \cX^{k-1}]$ and subtract $\cX^k$. Furthermore, when $k \ge \tau$, the array is updated as
$
[\cX^{j_{k,\tau}+1}, \ldots, \cX^k, \cX^{k+1}],
$
which removes the leftmost column and concatenates the new iterate on the right. 
These repeated extraction and concatenation operations introduce additional memory movement and communication overhead, which can limit practical efficiency even though the computational cost remains linear.

By contrast, Algorithm~\ref{G-S-based TKGK} stores orthogonal search directions rather than past iterates. Specifically, it maintains the orthogonal basis tensors
$ [\cU_{j_{k,\tau}}, \ldots, \cU_k],
$
generated using the Gram--Schmidt update
\[
\cU_{k+1}
=
\cD_{k+1}
-
\sum_{i=j_{k+1,\tau}}^{k}
\frac{\langle \cU_i,\cD_{k+1}\rangle}{\|\cU_i\|_F^2}\cU_i .
\]
This update depends only on the set $\{\cU_{j_{k+1,\tau}},\ldots,\cU_k\}$, not on their storage order. Therefore, when $k \ge \tau-1$, the implementation can overwrite the oldest basis tensor with the newly generated one, avoiding tensor extraction and concatenation altogether and thereby reducing communication and memory-traffic overhead while retaining linear-time complexity.

Finally, we note that the proposed Gram-Schmidt-based implementation not only reduces communication and memory overhead but also facilitates a natural connection between Algorithm~\ref{G-S-based TKGK} and Krylov subspace methods, as will be established in the next subsection.

\subsection{Connection to the Arnoldi-type Krylov subspace method}
\label{subsec_arnoldi}

In this subsection, leveraging the Gram-Schmidt-based implementation, we show that Algorithm~\ref{G-S-based TKGK} can recover the Arnoldi-type Krylov subspace method when either the incremental or the shuffle-once permutation strategy is used, and the truncation parameter is set to $\tau = \infty$.

For any $k\geq 0$, define 
\begin{equation}\label{g_pi_k}
	\begin{aligned}
		\cG_{\pi_{k}} := & 
		\sum^{m-1}_{i=1}\left(\cI - \cA_{\pi_{k, m}::}^{\dagger} \ast \cA_{\pi_{k, m}::}\right) 
		\ast \cdots \ast 
		\left(\cI - \cA_{\pi_{k, i+1}::}^{\dagger} \ast \cA_{\pi_{k, i+1}::}\right) 
		\ast
		\cA_{\pi_{k, i}::}^{\dagger} \ast \cB_{\pi_{k, i}::} \\
		& + \cA_{\pi_{k, m}::}^{\dagger} \ast \cB_{\pi_{k, m}::} ,
	\end{aligned}
\end{equation}
which, together with the definition of $\cT_{\pi_{k}}$ in \eqref{Tpi}, implies that  $\cP_{\pi_{k}} (\cdot) = \cT_{\pi_{k}} \ast (\cdot) + \cG_{\pi_{k}}$.
From Steps $5$ and $7$ in Algorithm~\ref{G-S-based TKGK}, we know that $\cX^{k+1} = \cX^k + \lambda_k \cU_k$ and $\cD_{k+1} = \cP_{\pi_{k+1}}\left(\cX^{k+1}\right) - \cX^{k+1}$. Thus, 
\begin{equation}\label{eq-0326-1-xie}
	\begin{aligned}
		\cD_{k+1} &=\cT_{\pi_{k+1}} \ast \cX^{k+1} + \cG_{\pi_{k+1}}-\cX^{k+1}\\
		&=\cT_{\pi_{k+1}}\ast\cX^k+\cG_{\pi_{k+1}}-\cX^k-\lambda_k(\cI - \cT_{\pi_{k+1}})\ast \cU_k\\
		&=\cP_{\pi_{k+1}}(\cX^{k})-\cX^k-\lambda_k(\cI - \cT_{\pi_{k+1}})\ast \cU_k.
	\end{aligned}
\end{equation}

If Algorithm~\ref{G-S-based TKGK} employs either the incremental or shuffle-once permutation strategy, then $\pi_k = \pi_0$ for all $k \ge 0$; that is, the permutation remains fixed throughout the iterations. Therefore, $\cP_{\pi_{0}}(\cX^{k})-\cX^k=\cD_k$, so \eqref{eq-0326-1-xie} simplifies to
\[
\cD_{k+1} =\cD_k-\lambda_k(\cI - \cT_{\pi_{0}})\ast \cU_k=\cD_k-\lambda_k \cC_{\pi_0}\ast \cU_k,
\]
where \(
\cC_{\pi_0} := \cI - \cT_{\pi_0}.
\)
Substituting the above expression into the Gram-Schmidt process in Step 8 of Algorithm~\ref{G-S-based TKGK} yields
\begin{equation}\label{Arnoldi_intermediate_step}
	\cU_{k+1} 
	= \cD_k - \sum_{i=0}^{k} \frac{\langle\cU_i, \cD_k\rangle}{\lVert\cU_i\rVert_F^2}\cU_i
	- \lambda_k\Big(\cC_{\pi_{0}} \ast \cU_k - \sum_{i=0}^{k} 
	\frac{\langle \cU_i, \cC_{\pi_{0}} \ast \cU_k \rangle}{\left\|\cU_i\right\|_F^2}\cU_i\Big).
\end{equation}
Since $\cD_k\in \spn\{\cU_0, \ldots, \cU_k\}$ and $\{\cU_0, \ldots, \cU_k\}$ forms an orthogonal basis, it follows that
$\cD_k - \sum_{i=0}^{k} \frac{\langle\cU_i, \cD_k\rangle}{\lVert\cU_i\rVert_F^2}\cU_i = 0$.
Thus, \eqref{Arnoldi_intermediate_step} simplifies to
\begin{align}\label{Arnoldi_linear_combination_form}
	\cC_{\pi_0} \ast \cU_k
	=
	\sum_{i=0}^{k}
	\frac{\langle \cU_i,\, \cC_{\pi_0} \ast \cU_k \rangle}
	{\|\cU_i\|_F^2}\cU_i
	-
	\frac{1}{\lambda_k}\cU_{k+1},
\end{align}
This demonstrates that   $\cC_{\pi_0} \ast \cU_k\in \spn\{\cU_0, \ldots, \cU_{k+1}\}$, thereby revealing the upper Hessenberg structure intrinsic to the Arnoldi process. 

Indeed, to further highlight the analogy with the classical Arnoldi decomposition~\cite[Section~10.5.1]{golub2013matrix}, we rewrite the above relation in a compact $t$-product form. For convenience, let us define
\[
\cU^k :=
\begin{bmatrix}
	\cU_0, \ldots, \cU_k
\end{bmatrix}
\in \mathbb{R}^{\ell \times (k+1)p \times n},
\]
and let $\cE_k \in \mathbb{R}^{(k+1)p \times p \times n}$ denote the tensor whose frontal slices are all zero except for the first one,  $\cE_{k_{::1}} = e_{k+1} \otimes I_p$,
where $\otimes$ denotes the Kronecker product. 
Moreover, we define the block Hessenberg tensor $\cH_k \in \mathbb R^{(k+1)p \times (k+1)p \times n}$
such that all frontal slices vanish except
\[
\cH_{k_{::1}} = H_k \otimes I_p,
\]
where $H_k \in \mathbb{R}^{(k+1)\times(k+1)}$ is the upper Hessenberg matrix generated by~\eqref{Arnoldi_linear_combination_form}, with entries
$h_{i,j} = \frac{\langle\cU_{i-1}, \cC_{\pi_0} \ast \cU_{j-1}\rangle}{\left\|\cU_{i-1}\right\|_F^2}$, 
$1 \le i \le j \le k+1$ and 
$h_{j+1,j} = -\frac{1}{\lambda_{j-1}}, 1 \le j \le k$,
$k \ge 1$. 
Let $\cR_{k} := h_{k+1, k}\cU_{k+1}$. 
Then, \eqref{Arnoldi_linear_combination_form} admits the compact $t$-product representation
\begin{align}\label{tensor_Arnoldi_decomposition}
	\cC_{\pi_0} \ast \cU^k
	=
	\cU^k \ast \cH_k + \cR_{k} \ast \cE_k^{\top}, 
\end{align}
which serves as the tensor analogue of the classical Arnoldi relation in~\cite[Section~10.5.1]{golub2013matrix} with the linear operator $\cC_{\pi_0}$, and it also coincides with the Arnoldi-type decomposition underlying the tensor $T$-global Arnoldi method proposed in~\cite{guide2020tensor}.

Therefore, when either the incremental or shuffle-once permutation strategy is used and the truncation parameter is set to \( \tau = \infty \), Algorithm~\ref{G-S-based TKGK} can be viewed as a Krylov subspace method driven by Arnoldi orthogonalization on the space \(\cX^0 + \cK_{k+1}\left(\cC_{\pi_0},\, \cG_{\pi_0} - \cC_{\pi_0} \ast \cX^{0} \right)\). Here,  given a linear operator $\cB $ and a tensor $\cR$, the order-$k$ Krylov subspace is defined (see \cite[Section 10.1.1]{golub2013matrix}) as
\(
\mathcal{K}_k(\cB, \cR) := \text{span}\{\cR, \cB \ast\cR, \cB^2\ast\cR, \ldots, \cB^{k-1}\ast\cR\}.
\)

Finally, let us present two remarks.  

\begin{remark}
	Note that we have already demonstrated that Algorithm~\ref{G-S-based TKGK} reduces to an Arnoldi-type Krylov subspace method when \( \tau = \infty \). In fact, for finite \( \tau \), the Gram-Schmidt-based implementation can be viewed as a truncated Krylov subspace method, where \( \tau \) serves as the truncation parameter determining the number of previous iterates used.
\end{remark} 

\begin{remark}
	Although the results in this paper have been developed within the tensor $t$-product framework, our approach and analysis can be readily extended to other classes of linear systems, provided an appropriate analogue of matrix-vector multiplication is preserved. In such cases, the methods proposed here can be adapted with minimal modifications.
\end{remark}

\section{Numerical experiments}\label{Sec_numer_exp}

In this section, we present numerical experiments on both synthetic and real-world datasets to validate our theoretical results.
On synthetic data, our experiments are specifically designed to evaluate the performance of our 
Gram-Schmidt-based implementation (Algorithm~\ref{G-S-based TKGK}), which we refer to as GS-TKGK. 
We first compare the computational efficiency of the proposed method with 
the tensor extension of the tridiagonal-based method~\cite{rieger2023generalized}, 
as analyzed in Section~\ref{section4-1}. 
We then systematically examine the impact of different truncation parameters on convergence behavior.
On a real-world video dataset, we further compare Algorithm~\ref{G-S-based TKGK} with two state-of-the-art algorithms, 
namely, the TK method and the tensor average Kaczmarz method with stochastic heavy ball momentum (tAKSHBM) method~\cite{liao2024accelerated,zeng2024adaptive}.
All the experiments are conducted on MATLAB R2025a for Windows 11 on a LAPTOP PC with an Intel Core Ultra 7 255H @ 2.00 GHz and 32 GB RAM.
The code to reproduce our results can be found at \href{https://github.com/xiejx-math/TKGK-codes}{https://github.com/xiejx-math/TKGK-codes}.

\subsection{Synthetic data}
In this section, we generate synthetic tensors $\cA \in \mathbb{R}^{m \times \ell \times n}$ as follows. 
Given target rank $r$ and a prescribed condition number $\kappa > 1$ for frontal slices, 
we construct
\(
\cA(:, :, i) = U_i D_i V_i^\top, i = 1, 2, \ldots, n,
\)
where $U_i \in \mathbb{R}^{m \times r}$, $D_i \in \mathbb{R}^{r \times r}$, and $V_i \in \mathbb{R}^{\ell \times r}$.  
In MATLAB notation, these are obtained by $\tt{[U_i, \sim] = qr(randn(m, r), 0)}, \tt{[V_i, \sim] = qr(randn(\ell, r), 0)}$, and $\tt{D_i = diag(1 + (kappa - 1) .* rand(r, 1))}.$
This ensures that the rank and condition number of each frontal slice are bounded above by $r$ and $\kappa$, respectively.
To form a consistent tensor linear system $\cA \ast \cX = \cB$, we generate the ground-truth solution as $\cX^* = \tt{randn(\ell, p, n)}$ and set $\cB = \cA \ast \cX^*$. 
All algorithms start from $\cX^0 = 0$, 
and terminate the iteration when the
\(
\mathrm{RSE} = \frac{\|\cX^k - \cA^{\dagger} \ast \cB\|_F^2}{\|\cX^0 - \cA^{\dagger} \ast \cB\|_F^2}
\)
falls below a prescribed tolerance, or when the maximum number of iterations is reached.
Each experiment is independently repeated 10 times to capture the variability due to random data generation. 
In all subsequent figures, the central line denotes the median across trials, the dark shaded band corresponds to the interquartile range between the 25th to 75th percentiles, and the light shaded region covers the full range from the minimum to the maximum observed values.

\subsubsection{Comparison to the tridiagonal-based implementation}

In this subsection, we compare GS-TKGK with a tridiagonal-based implementation (Tri-TKGK) adapted from \cite[Section 6]{rieger2023generalized}, focusing solely on Shuffle-once permutation strategy, specifically GS-TKGK-SO and Tri-TKGK-SO.
Both methods adopt update schemes for $\mathcal{X}^k$ derived from~\eqref{RIM_subspace}, yet differ in their realization.
In contrast to the Gram-Schmidt process employed in our work,
the Tri-TKGK-SO computes the solution $s^k$ of linear system~\eqref{eq-s_k} via the expression~\eqref{oversk_undersk}. 
To reduce the cost of computing $(V_k^\top V_k)^{-1} \in \mathbb{R}^{(k-j_{k,\tau})\times(k-j_{k,\tau})}$ in~\eqref{oversk_undersk}, Tri-TKGK-SO exploits the explicit tridiagonal structure of the inverse, whose non-zero entries are given by
\[
\begin{aligned}
	((V_k^\top V_k)^{-1})_{i,i} &= 
	\begin{cases}
		\bigl(\gamma_{j_{k,\tau}}\,\underline{s^{\,j_{k,\tau}}}\bigr)^{-1}, & i = 1, \\[6pt]
		\bigl(\gamma_{j_{k,\tau}+i-2}\,\underline{s^{\,j_{k,\tau}+i-2}}\bigr)^{-1}
		+
		\bigl(\gamma_{j_{k,\tau}+i-1}\,\underline{s^{\,j_{k,\tau}+i-1}}\bigr)^{-1}, & i = 2, \dots, k-j_{k,\tau},
	\end{cases} \\[8pt]
	((V_k^\top V_k)^{-1})_{i,i+1} 
	&= ((V_k^\top V_k)^{-1})_{i+1,i} 
	= -\bigl(\gamma_{j_{k,\tau}+i-1}\,\underline{s^{\,j_{k,\tau}+i-1}}\bigr)^{-1}, \quad i = 1, \dots, k-j_{k,\tau}-1.
\end{aligned}
\]

Figure~\ref{fig_comparison_GS_TRI} shows that both implementations require identical iteration counts for all tested $\kappa$ and $r$, 
which is consistent with their mathematical equivalence. 
However, GS-TKGK-SO consistently achieves lower CPU time.
This performance difference is consistent with the theoretical complexity analysis in Section~\ref{section4-1}.
Tri-TKGK-SO involves repeated extraction and concatenation operations on tensor arrays, 
incurring additional memory movement, 
whereas GS-TKGK-SO performs in-place updates based on orthogonal bases, thereby reducing computational overhead.

\begin{figure}[h!]
	\centering
	\includegraphics[width=0.32\linewidth]{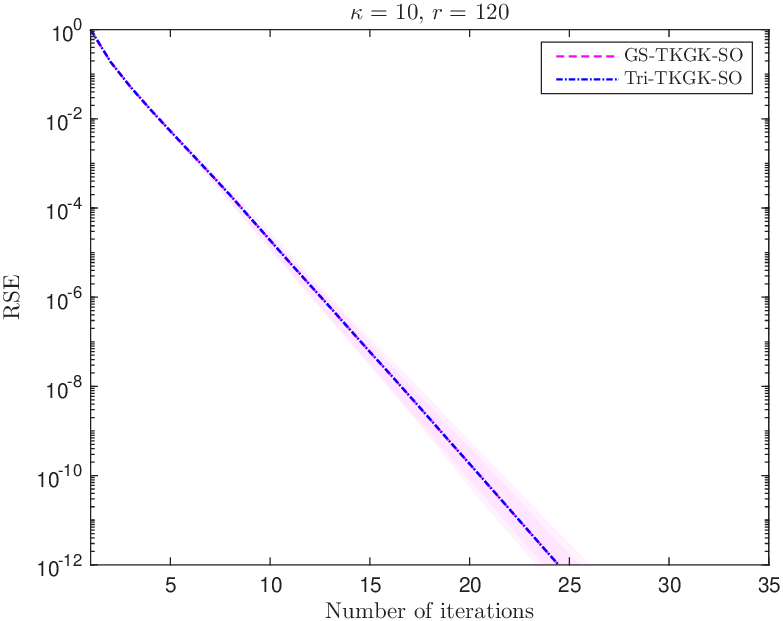}
	\includegraphics[width=0.32\linewidth]{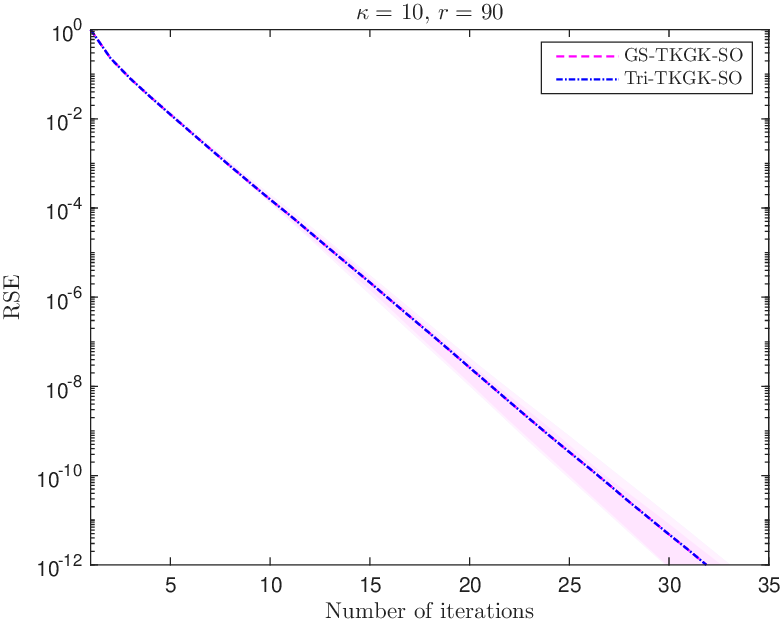}
	\includegraphics[width=0.32\linewidth]{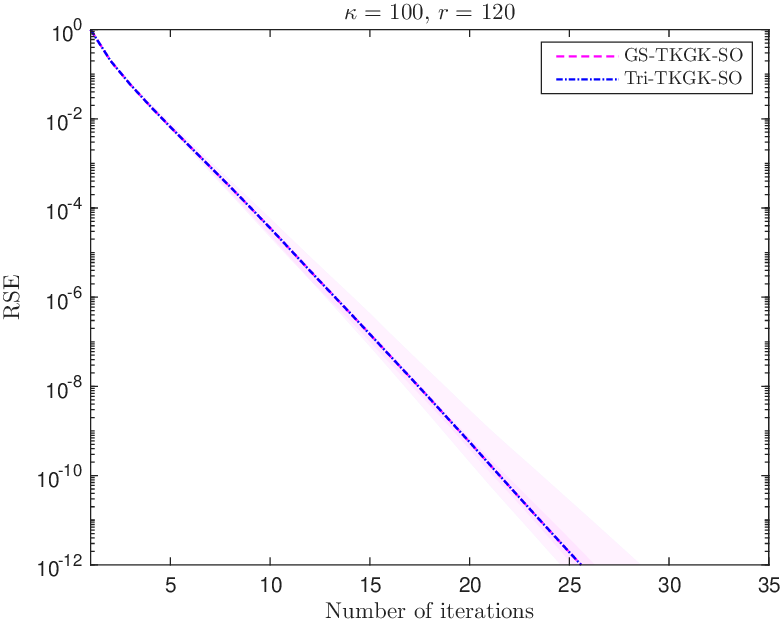}\\
	\includegraphics[width=0.32\linewidth]{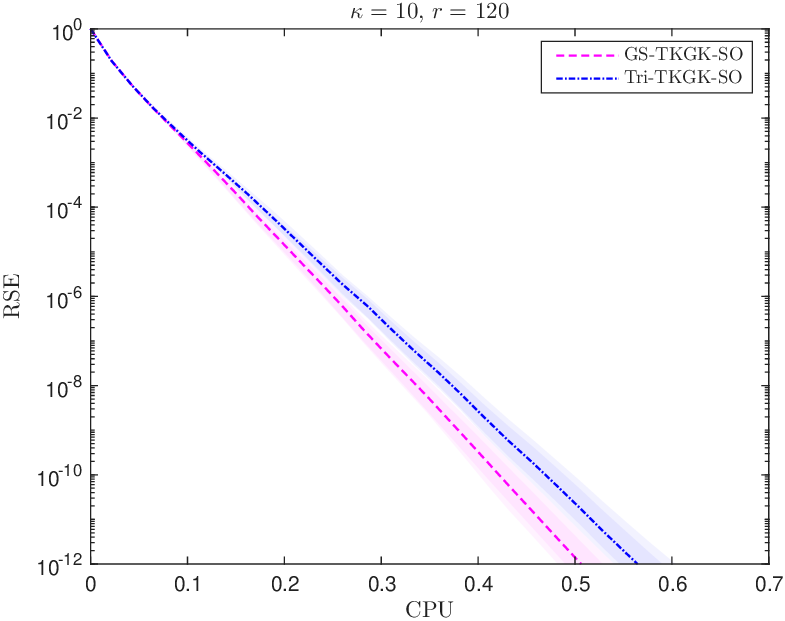}
	\includegraphics[width=0.32\linewidth]{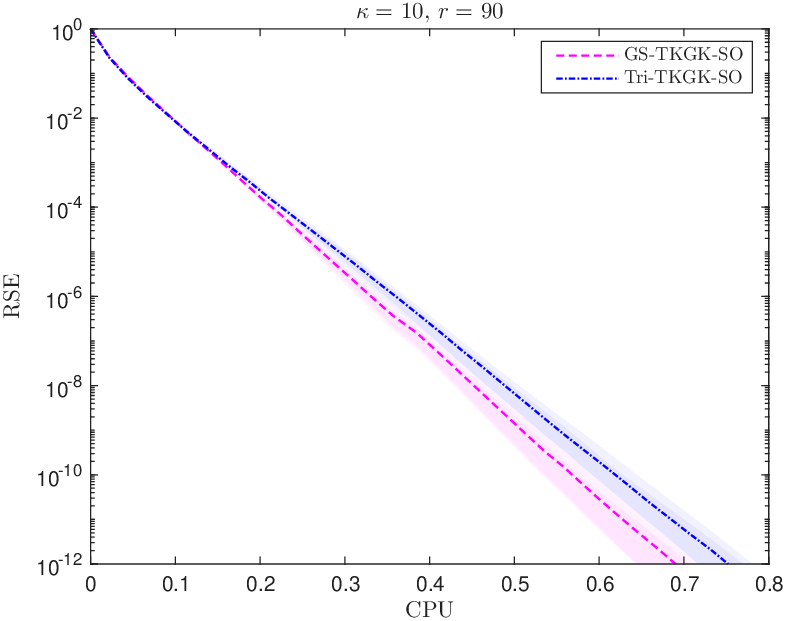}
	\includegraphics[width=0.32\linewidth]{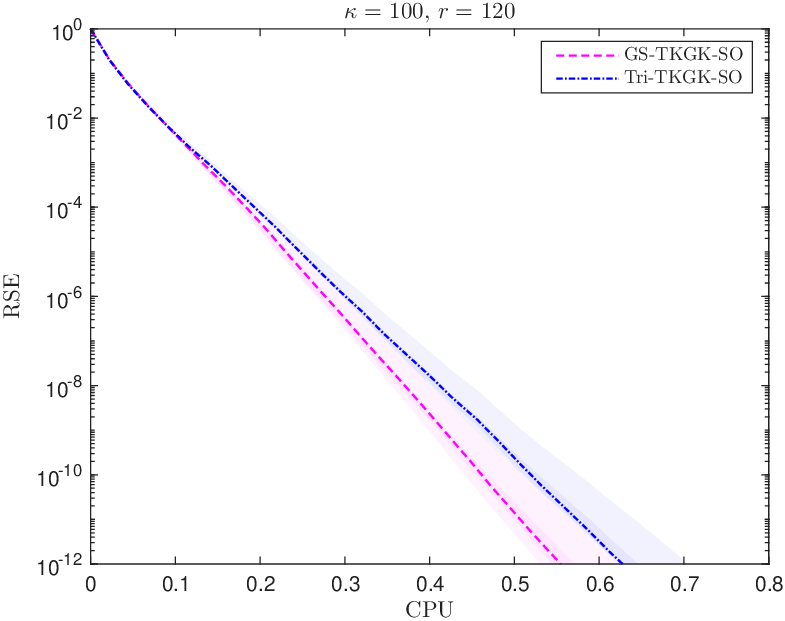}
	\caption{Comparison of iteration numbers (top) and CPU time (bottom) with respect to 
		$\kappa$ and $r$ for the Gram--Schmidt-based and tridiagonal-based implementations. 
		The title of each subplot indicates the corresponding values of $\kappa$ and $r$, 
		with other parameter fixed as $m = 200, \ell = 120, n = 3, p = 120$, and $\tau = 10$.
		All experiments terminate once the RSE $< 10^{-12}$.}
	\label{fig_comparison_GS_TRI}
\end{figure}

\subsubsection{Choices of the truncation parameter $\tau$}
In this subsection, we investigate the impact of the truncation parameter $\tau$ on the convergence behavior of the GS-TKGK-SO. 
The performance of the algorithms is measured in both the CPU time and the number of iterations.
The results are displayed in Figure \ref{fig_comparison_t}.

As shown in Figure~\ref{fig_comparison_t}, the number of iterations decreases monotonically as the truncation parameter $\tau$ increases over the tested values $\tau = 1, 2, 3, 5,$ and $10$. 
In contrast, the CPU time decreases for $\tau = 1, 2, 3,$ and $5$, but remains nearly unchanged when $\tau$ is increased from $5$ to $10$. 
This indicates that $\tau = 5$ achieves a good balance between the number of iterations and CPU time.

\begin{figure}[h!]
	\centering
	\includegraphics[width=0.32\linewidth]{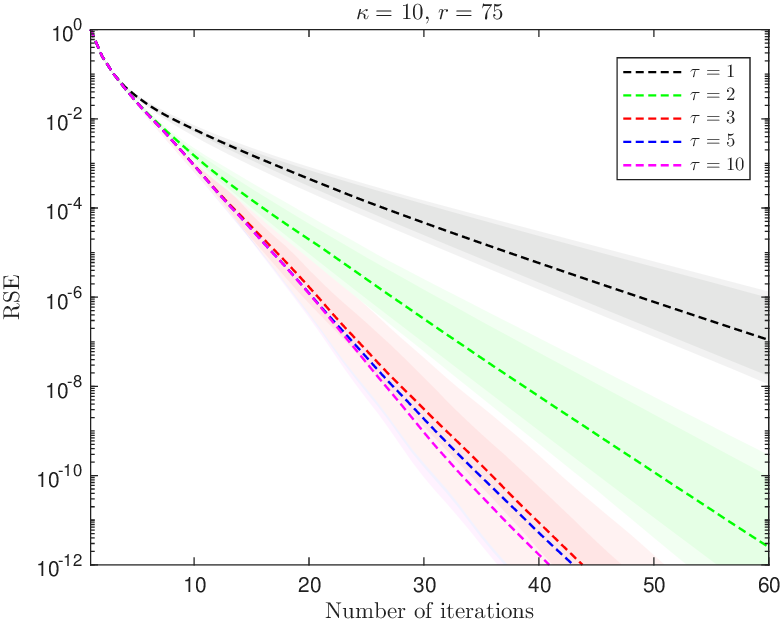}
	\includegraphics[width=0.32\linewidth]{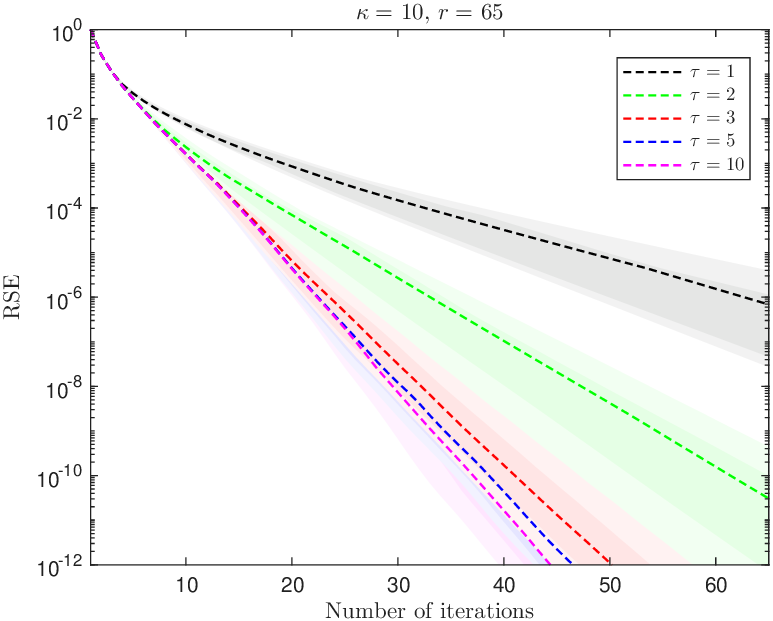}
	\includegraphics[width=0.32\linewidth]{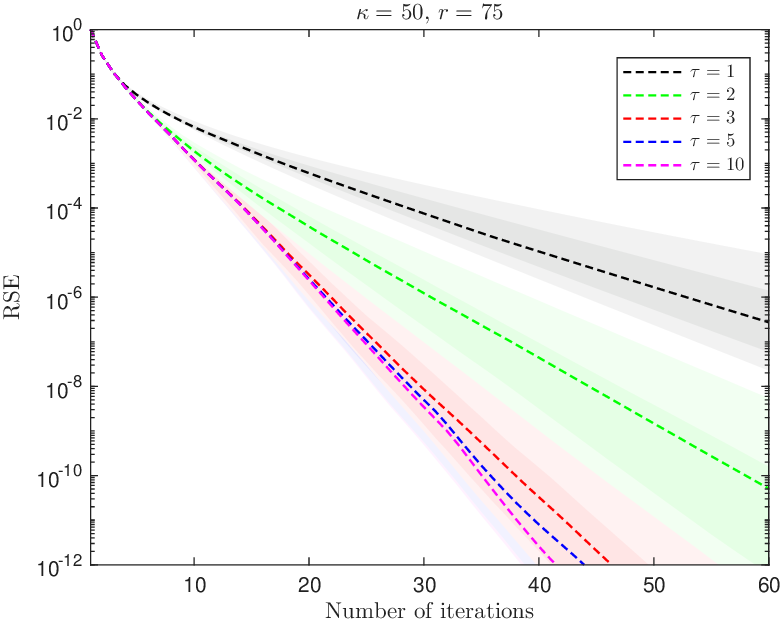}\\
	\includegraphics[width=0.32\linewidth]{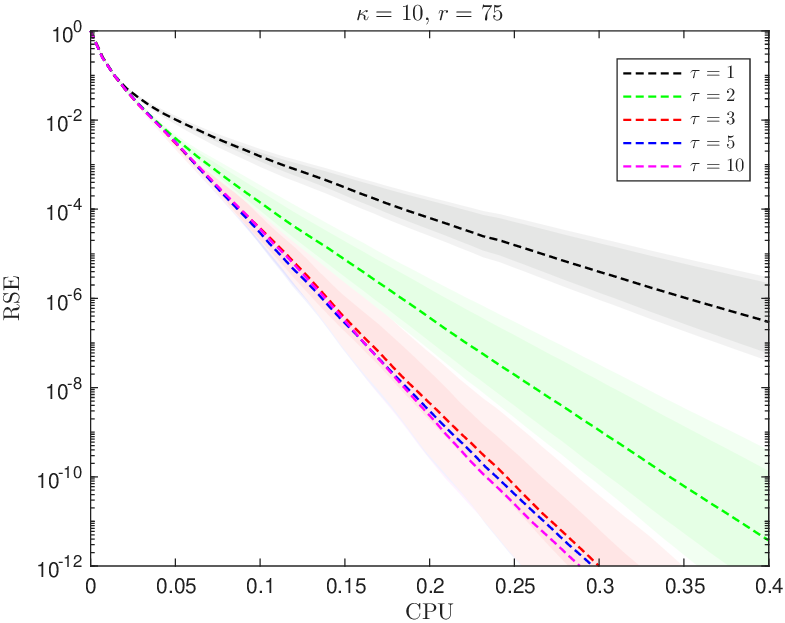}
	\includegraphics[width=0.32\linewidth]{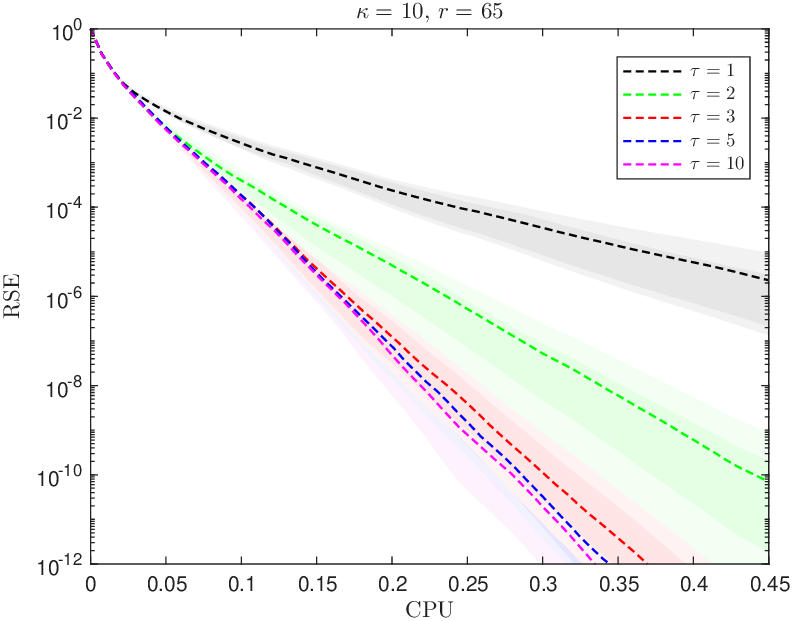}
	\includegraphics[width=0.32\linewidth]{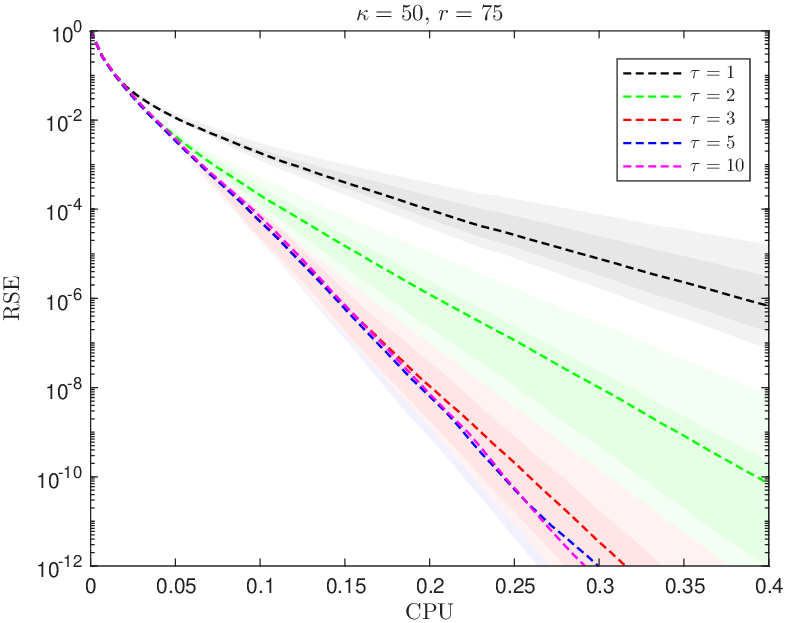}
	\caption{The figures illustrate the evolution of RSE with respect to the
		number of iterations (top) and the CPU time (bottom). The title of each
		subplot indicates the corresponding values of $\kappa$ and $r$. We set $m=100, \ell = 75, n = 3$, and $p = 75$.
		All experiments terminate once the RSE $< 10^{-12}$.}
	\label{fig_comparison_t}
\end{figure}

\subsection{Real world data}
In this section, we employ the built-in MATLAB video dataset \texttt{traffic.avi} as the ground truth to simulate the deblurring process.
%
%
We use the blur tensor $\cA$ from \cite{reichel2022tensor}, whose frontal slices are given by
\[
\mathcal{A}_{::j} = M_2(j,1) M_1, \quad j = 1,\dots,n,
\]
where $M_1$ and $M_2$ are Toeplitz matrices defined in MATLAB notation as
\[
\tt{M_1 = \frac{1}{\sqrt{2\pi\sigma}} toeplitz(z_1), \ M_2 = \frac{1}{\sqrt{2\pi\sigma}} toeplitz(z_1,z_2)},
\]
with 
\[
\left\{
\begin{aligned}
	\tt{z_1} &= \tt{\left[ exp\left( -\left( (0:\texttt{band}-1).^2/(2\sigma^2) \right) \right),\; \texttt{zeros}(1, \ell-\texttt{band}) \right], }\\
	\tt{z_2} &= [\tt{z_1(1)\cdot \texttt{fliplr}(z_1(2:\texttt{end}))]}.
\end{aligned}
\right.
\]
Here, \texttt{band} denotes the bandwidth of the Toeplitz blocks and $\sigma > 0$ is the standard deviation of the Gaussian point-spread function, specifically, larger values of $\sigma$ correspond to stronger blurring. 
Throughout the experiments, we set \texttt{band} $= 6, \sigma = 1.8$.
The observed blurred tensor $\mathcal{B}$ is then computed as $\mathcal{B} = \mathcal{A} * \mathcal{X}^{*}$, where $\mathcal{X}^{*}$ denotes the \texttt{traffic.avi} video stored as a $120 \times 160 \times 120$ third order tensor. 
All algorithms initialize with $\cX^0 = 0$.
To evaluate the quality of the reconstructed videos, we compute the peak signal-to-noise ratio (PSNR) and structural similarity index measure (SSIM) between the ground truth and the recovered results using MATLAB's built-in functions \texttt{psnr} and \texttt{ssim}.

In the following experiments, we evaluate three sampling strategies implemented within the GS-TKGK framework described in Algorithm~\ref{G-S-based TKGK}, namely GK-TKGK-RR, GS-TKGK-SO, and GS-TKGK-IS, with truncation parameter $\tau = 5$. 
For comparison, we also consider the TK method under the same three sampling strategies (TK-RR, TK-SO, and TK-IS) as well as the tAKSHBM method.
Specifically, the tAKSHBM method iterate as
\begin{equation*}
	\cX^{k+1} = \cX^{k} - \alpha_k \nabla f_{S_k}(\cX^{k}) + \beta_k (\cX^{k} - \cX^{k-1}),
\end{equation*}
where 
$\nabla f_{S_k}(\cX^{k}) = \mathcal{A}_{\tau_{i_k}::}^T \ast (\mathcal{A}_{\tau_{i_k}::} \ast \cX^{k} - \mathcal{B}_{\tau_{i_k}::})$, 
and $\tau_{i_k} \subseteq [m]$ with $|\tau_{i_k}| = q$ is selected with probability $\|\mathcal{A}_{\tau_{i_k}::}\|_F^2 / \|\mathcal{A}\|_F^2$.
Here, $\alpha_k$ and $\beta_k$ are computed by minimizing 
$\lVert \cX^{k} - \alpha \nabla f_{S_k}(\cX^{k}) + \beta (\cX^{k} - \cX^{k-1})-\cA^\dagger \ast \cB\rVert_F^2$.
The performance of the algorithms is evaluated in terms of CPU time and the number of full iterations, with the latter ensuring a uniform count of operations across all algorithms when traversing the horizontal slices of the tensor $\cA$. 
Specifically, for the TKGK and TK algorithms, the number of full iterations is equivalent to the number of algorithmic iterations.
For tAKSHBM, we adopt the block size $q=15$ as used in~\cite[Example 4]{liao2024accelerated},
and its number of full iterations is defined as the iterations count multiplied by the block size $q$ and divided by the number of horizontal slices $m$. 

Figure~\ref{fig_video_tK_GSTKGK_tAKSHBM} presents the convergence behaviors of all the algorithms over $2000$ full iterations.
It shows that GS-TKGK framework, particularly GS-TKGK-SO, consistently outperforms both TK and tAKSHBM in terms of full iterations and CPU time at termination. 
This indicates that GS-TKGK improves both convergence efficiency and reconstruction accuracy for tensor-based deblurring problems.

\begin{figure}[h!]
	\centering
	\includegraphics[width=0.48\linewidth]{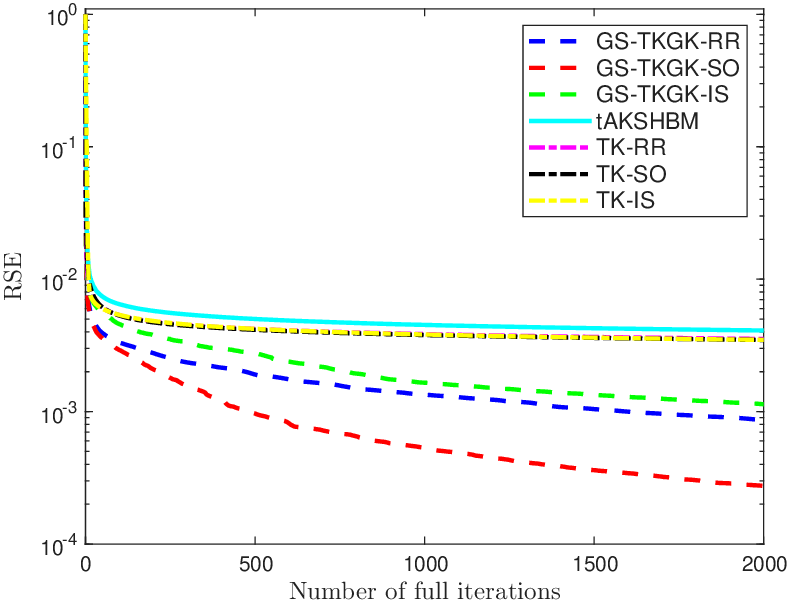} 
	\includegraphics[width=0.48\linewidth]{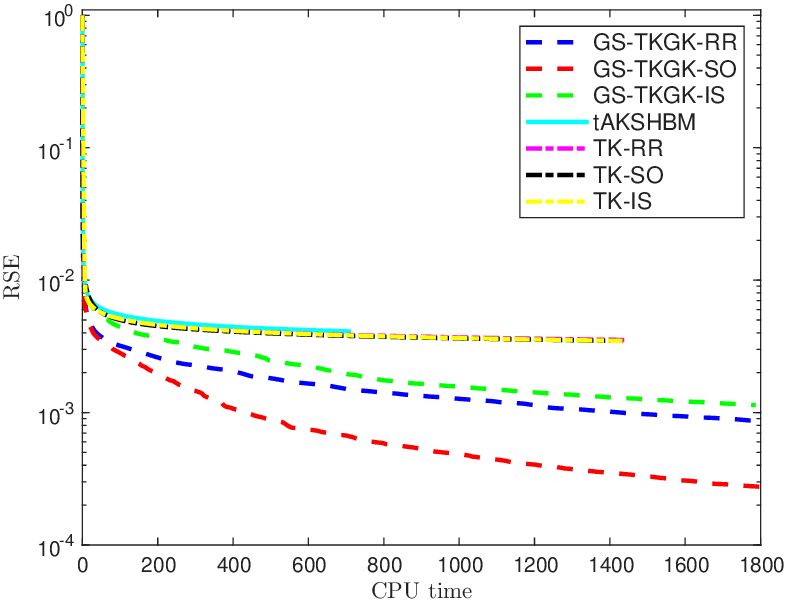}
	\caption{Comparison of RSE versus number of full iterations (left) and CPU time (right) for 
		TK, GS-TKGK, and tAKSHBM on blurred \texttt{traffic.avi} video reconstruction. 
		The algorithm terminates after at most 2000 full iterations.}
	\label{fig_video_tK_GSTKGK_tAKSHBM}
\end{figure}


Table~\ref{tab:video_method_comparison} summarize the computational results and reconstruction performance of all the algorithms when terminated at RSE $< 5 \times 10^{-3}$.
It indicates that, 
under the same stopping criterion, the proposed GS-TKGK method is substantially more efficient than both TK and tAKSHBM in terms of the number of full iterations and the CPU time, while all methods achieve nearly indistinguishable reconstruction quality in terms of PSNR and SSIM.
For both TK and GS-TKGK, the permutation strategy has a clear effect on convergence speed but only a minor impact on reconstruction accuracy. Among the three strategies, SO yields the fastest convergence, IS requires more iterations and computational effort but attains slightly better image quality, and RR lies in between. 
Figure \ref{fig_video_reconstruction} shows the 20th, 51st, 72nd, and 100th original frames, 
the corresponding blurred observation, and the images recovered by tAKSHBM, GS-TKGK-SO and TK-SO. 

\begin{table}[h]
	\centering
	\small 
	\setlength{\tabcolsep}{4pt} 
	\caption{Performance results including the full iterations, CPU time, and the average PSNR/SSIM over 120 frames. 
		All experiments terminate once the RSE $< 5 \times 10^{-3}$.}
	\begin{tabular}{lcccc} 
		\toprule
		{Methods} & \makecell{Number of full \\ iterations} & {CPU} & {PSNR} & {SSIM} \\
		\midrule
		GS-TKGK-RR & 21 & 15.853 & 28.9672 & 0.9166 \\
		GS-TKGK-SO  & \textbf{16} & \textbf{12.180} & 28.9532 & 0.9126  \\
		GS-TKGK-IS &  79 & 61.618 & \textbf{28.9712} & 0.9281 \\
		tAKSHBM  & 499 & 169.517 & 28.9423 &  0.9252  \\
		TK-RR & 137 & 92.170 & 28.9162 & 0.9236\\
		TK-SO & 135 & 89.198 & 28.9182 & 0.9211\\
		TK-IS & 161 & 106.234 & 28.9210 & \textbf{0.9291}\\
		\bottomrule
	\end{tabular}
	\label{tab:video_method_comparison}
\end{table}

\begin{figure}[h!]
	\centering
	\includegraphics[width=0.96\linewidth]{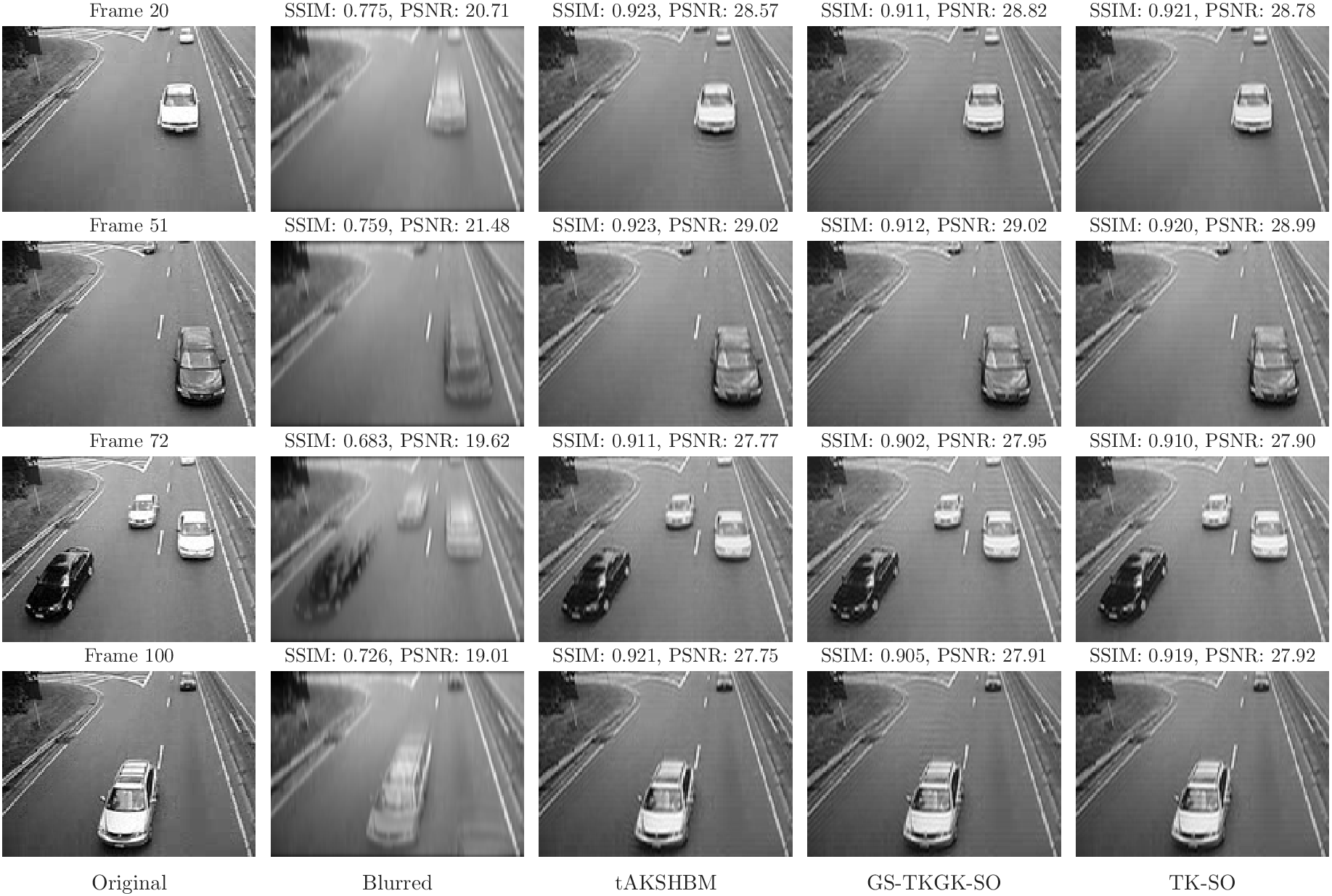} 
	\caption{Blurred \texttt{traffic.avi} video reconstruction results for 
		tAKSHBM, GS-TKGK-SO and TK-SO on frames indexed by $20, 51, 72,$ and $100$. 
		The title of each subplot (except the first column of original frames) reports the PSNR and SSIM values of the recovered or blurred frame with respect to the corresponding original frame.
		The algorithms terminate if RSE $< 5\times 10^{-3}$.}
	\label{fig_video_reconstruction}
\end{figure}

\section{Concluding remarks}
\label{Sec_conclu_rmk}

We have extended and analyzed the Gearhart-Koshy accelerated Kaczmarz methods, incorporating RR, SO, and IS sampling strategies, for application to tensor linear systems. We demonstrated that the proposed TKGK method exhibits linear convergence, with a convergence factor superior to that of the plain Kaczmarz method. By exploiting the specific structure of the problem, we proposed an efficient Gram-Schmidt-based implementation to compute the next iterate in linear time. Building on this implementation,  we established a connection between the Gearhart-Koshy accelerated Kaczmarz method and Arnoldi-type Krylov subspace methods when either the SO or IS permutation strategy is employed and the truncation parameter is set to $\tau = \infty$. Numerical results confirm the efficiency of the proposed method.
 
 
 It is usual in practical applications to encounter the problem of solving the tensor linear system $\cA \ast \cX = \cB$ subject to additional optimality criteria, such as sparsity or low-rank \cite{chen2021regularized,zeng2026stochastic,jeong2025linear,Schopfer2019Linear}. The extension of the Gearhart-Koshy acceleration to obtain solutions satisfying these optimality criteria represents a valuable avenue for future research. Moreover, tensor linear systems arising in practical problems are likely to be inconsistent due to noise \cite{zeng2025adaptive,du2019tight,zouzias2013randomized,MaConvergence15}, making the extension of the TKGK method to handle inconsistent systems another important avenue for further investigation.
 

\bibliographystyle{abbrv}
\bibliography{ref}

\section{Appendix. Proof of the main results}

\subsection{Proof of Theorem~\ref{trrk_convergence_thm}}

	\begin{proof}[Proof of Theorem~\ref{trrk_convergence_thm}]
		To establish the results, we first note that all iterates of Algorithm~\ref{tensor-k} belong to the affine subspace $\cX^0 +\operatorname{range}_K(\cA^{\dagger})$.
		Indeed, for any $1 \le i \le m$, we have
		\[
		\operatorname{range}_K(\cA_{i::}^\dagger) = \operatorname{range}_K(\cA_{i::}^\top) 
		\subseteq \operatorname{range}_K(\cA^\top) = \operatorname{range}_K(\cA^\dagger),
		\] 
		where the two equalities follows from~\eqref{range_Adagger_eq_range_Atop} 
		and the inclusion follows from the fact that $\cA_{i::}^{\top} = \cA^{\top} \ast \cI_{i::}$.
		Consequently, Algorithm~\ref{tensor-k} ensures that
		\(
		\cX^k \in \cX^0 + \operatorname{range}_K(\cA^\dagger)
		\). 
		Besides, $\cX_*^0 = \cA^{\dagger} \ast \cB + (\cI - \cA^{\dagger} \ast \cA) \ast \cX^0
		= \cX^0 + \cA^{\dagger} \ast (\cB - \cA \ast \cX^0) \in \cX^0 + \operatorname{range}_K(\cA^\dagger)$,
		therefore $\cX^k-\cX_*^0 \in \operatorname{range}_K(\cA^{\dagger})$. 
		Since $\cA^{\dagger} \ast \cA$ is the projector onto $\operatorname{range}_K(\cA^{\dagger})$, it follows
		\begin{equation}\label{eq:proj_identity}
			\cA^\dagger \ast \cA \ast (\cX^k-\cX_*^0) = \cX^k-\cX_*^0 .
		\end{equation}
		By the definitions of $\cT_{\pi_{k}}$, $\cP_{\pi_{k}}$, 
		and $\cG_{\pi_{k}}$ in \eqref{Tpi}, \eqref{full-projection-oper}, and \eqref{g_pi_k}, 
		we know that one epoch of Algorithm~\ref{tensor-k} can be equivalently rewritten as 
		$\cX^{k+1} = \cP_{\pi_{k}}(\cX^k) = \cT_{\pi_{k}} \ast \cX^k + \cG_{\pi_{k}}$.
		Hence we have
		\begin{equation}\label{sun_0404}
			\begin{aligned}
				\lVert\cX^{k+1} - \cX_*^0\rVert_F^2 &= \lVert\cP_{\pi_{k}}(\cX^k) - \cX_*^0\rVert_F^2 
				= \lVert\cT_{\pi_k}\ast(\cX^k-\cX_*^0)\rVert_F^2 \\
				& = \lVert\cT_{\pi_k}\ast\cA^\dagger\ast\cA\ast(\cX^k-\cX_*^0)\rVert_F^2 \\ 
				& = \lVert\bc(\cT_{\pi_k} \ast \cA^\dagger \ast \cA)\uf(\cX^k-\cX_*^0)\rVert_F^2 \\
				& \le \lVert\bc(\cT_{\pi_k} \ast \cA^\dagger \ast \cA)\rVert_2^2 \, \lVert\uf(\cX^k-\cX_*^0)\rVert_F^2 \\
				& = \lVert\bc(\cT_{\pi_k} \ast \cA^\dagger \ast \cA)\rVert_2^2 \, \lVert\cX^k-\cX_*^0\rVert_F^2
			\end{aligned}
		\end{equation}
		where the third equality follows from~\eqref{eq:proj_identity}.
		
		We next prove $\rho_k = \left\|\bc(\cT_{\pi_k} \ast \cA^\dagger \ast \cA)\right\|_2^2 < 1$. 
		By \cite[Lemma~1 (iii)]{lund2020tensor}, we have
		\[
		\bc(\cT_{\pi_k}\ast \cA^\dagger\ast \cA)
		=\bc(\cT_{\pi_k})\bc(\cA^\dagger)\bc(\cA),
		\]
		The idea of the proof is similar to that of Lemma 1 in \cite{han2025simple}. Specifically, 
		it suffices to prove that $\lVert\bc(\cT_{\pi_k})\bc(\cA^\dagger)\bc(\cA)z\rVert_2^2 < \lVert z\rVert_2^2$ for any $z \in \mathbb{R}^{\ell \times p \times n}$, $ z \neq 0$. 
		Here we note that for vectors, $\lVert\cdot\rVert_2$ denotes the Euclidean norm. 
		If $\bc(\cA^\dagger)\bc(\cA)z = 0$ the inequality is already satisfied.
		If $\bc(\cA^\dagger)\bc(\cA)z \neq 0$, then
		$$\bc(\cA)\left(\bc(\cA^\dagger)\bc(\cA)z\right) \neq 0. $$
		Consequently, there exists a certain $i_0 \in [m]$ such that $\bc(\cA_{\pi_{k,i_{0}}::})\bc(\cA^{\dagger})\bc(\cA)z \neq 0$, implying
		\begin{align*}
			\lVert&\bc(\cI - \cA_{\pi_{k,i_{0}}::}^{\dagger}\cA_{\pi_{k,i_{0}}::})\bc(\cA^{\dagger})\bc(\cA)z\rVert_2^2 \\
			&= \lVert\bc(\cA^{\dagger})\bc(\cA)z\rVert_2^2 - \lVert\bc(\cA_{\pi_{k,i_{0}}::}^{\dagger}\cA_{\pi_{k,i_{0}}::})\bc(\cA^{\dagger})\bc(\cA)z\rVert_2^2 \\
			& < \lVert\bc(\cA^{\dagger})\bc(\cA)z\rVert_2^2 \le \lVert z\rVert_2^2.
		\end{align*}
		Therefore, we obtain 
		\begin{align*}
			\lVert&\bc(\cT_{\pi_k})\bc(\cA^\dagger)\bc(\cA)z\rVert_2^2 \\
			&= \lVert\bc(\cI - \cA_{\pi_{k,m}::}^{\dagger}\cA_{\pi_{k,m}::}) 
			\cdots \bc(\cI - \cA_{\pi_{k,i_0}::}^{\dagger}\cA_{\pi_{k,i_0}::})
			\bc(\cA^\dagger)\bc(\cA)z\rVert_2^2 \\
			& \le \lVert\bc(\cI - \cA_{\pi_{k,m}::}^{\dagger}\cA_{\pi_{k,m}::})\rVert_2^2 
			\cdots \lVert\bc(\cI - \cA_{\pi_{k,i_0}::}^{\dagger}\cA_{\pi_{k,i_0}::})\bc(\cA^\dagger)\bc(\cA)z\rVert_2^2 \\
			& \le \lVert\bc(\cI - \cA_{\pi_{k,i_0}::}^{\dagger}\cA_{\pi_{k,i_0}::})\bc(\cA^\dagger)\bc(\cA)z\rVert_2^2 < \lVert z \rVert_2^2
		\end{align*}
		as desired. This completes the proof the theorem.
	\end{proof}

	
	\subsection{Proof of Lemma~\ref{equiv} and~\ref{ensure-comput}}
	
	\begin{proof}[Proof of Lemma~\ref{equiv}]
		First, we prove ``\(
		\cA \ast \cX^k = \cB 
		\Rightarrow 
		r_{\pi_k}(\cX^k) = 0\)''. For any $1\le i\le m$, we have
		\[
		\cP_{\pi_{k,i}}(\cX^k) = \cX^k-\cA_{\pi_{k,i}::}^\dagger\ast(\cA_{\pi_{k,i}::}\ast\cX^k-\cB_{\pi_{k,i}::})=\cX^k,
		\]
		which yields $r_{\pi_k}(\cX^k)=0$ by the definition of $r_{\pi_k}(\cX^k)$ in~\eqref{def_r_k}.
		
		Next, we prove ``$r_{\pi_k}(\cX^k) = 0
		\Rightarrow
		\cP_{\pi_k}(\cX^k) = \cX^k$''. 
		By the definition of $r_{\pi_k}(\cX^k)$ in~\eqref{def_r_k}, the condition $r_{\pi_k}(\cX^k) = 0$ implies
		\[
		\cA_{\pi_{k,i}}^{\dagger} \ast \left(\cA_{\pi_{k,i}} \ast \cP_{\pi_{k,i-1}} \circ \cdots \circ \cP_{\pi_{k,1}}(\cX^k) - \cB_{\pi_{k,i}}\right)=0,\  1 \le i \le m,
		\]
		and therefore \(
		\cP_{\pi_{k,i}} \circ \cdots \circ \cP_{\pi_{k,1}}(\cX^k)
		= \cP_{\pi_{k,i-1}} \circ \cdots \circ \cP_{\pi_{k,1}}(\cX^k).
		\)
		Consequently, using the definition of $\cP_{\pi_k}(\cX^k)$ in~\eqref{full-projection-oper}, we obtain $$\cP_{\pi_k}(\cX^k)=\cP_{\pi_{k,m}} \circ \cdots \circ \cP_{\pi_{k,1}}(\cX^k)=\cP_{\pi_{k,m-1}} \circ \cdots \circ \cP_{\pi_{k,1}}(\cX^k)=\cdots=\cX^k.$$
		
	Then, we prove	``$\cP_{\pi_k}(\cX^k) = \cX^k
		\Rightarrow
		\cA \ast \cX^k = \cB$''. As in~\cite[Lemma~1]{rieger2023generalized}, 
		we have the following identity
		\begin{align}\label{successive_Pythagoras}
			\|\cX^k-\cX_*^0\|_F^2
			=
			\|\cP_{\pi_k}(\cX^k)-\cX_*^0\|_F^2
			+
			\|r_{\pi_k}(\cX^k)\|_F^2 .
		\end{align}
		Since $\cP_{\pi_k}(\cX^k)=\cX^k$, it follows that $\|r_{\pi_k}(\cX^k)\|_F^2=0$, 
		which implies $\cA \ast \cX^k = \cB$. 
		
		We now prove that for all $k \geq 0$,  
		\(
		\mathcal{A} \ast \mathcal{X}^{k} \neq \mathcal{B}\)
		implies
		\(\mathcal{A} \ast \mathcal{X}^{i} \neq \mathcal{B}\) for \(i = 0,1,\dots,k.\)
		The case $k = 0$ is immediate.  
		For $k \geq 1$, it suffices to show that
		\(
		\mathcal{A} \ast \mathcal{X}^{k} \neq \mathcal{B}\)
		implies
		\(\mathcal{A} \ast \mathcal{X}^{k-1} \neq \mathcal{B}.
		\)
		Suppose, for contradiction, that $\mathcal{A} \ast \mathcal{X}^{k} \neq \mathcal{B}$ but $\cA \ast \cX^{k-1} = \cB$.
		When $k = 1$, the contradiction hypothesis gives $\mathcal{A} \ast \mathcal{X}^{0} = \mathcal{B}$, so Lemma~\ref{equiv} yields $\mathcal{X}^{0} = \mathcal{P}_{\pi_{0}}(\mathcal{X}^{0})$.
		Then, we obtain $\Pi_{0}=\operatorname{aff}\{\cX^0,P_{\pi_0}(\cX^0)\}=\{\cX^0\}$.
		From the definition of $\cX^1$ in \eqref{RIM_subspace}, it holds that $\cX^1=\cX^0$.
		This implies $\cA \ast \cX^1 = \cB$ provided $\cA \ast \cX^0 = \cB$, which contradicts the assumption that $\cA \ast \cX^1 \neq \cB$.
		For $k\geq2$, we have $\cP_{\pi_{k-1}}(\cX^{k-1}) = \cX^{k-1}$ due to $\cA \ast \cX^{k-1} = \cB$ by Lemma \ref{equiv} and the search subspace $\Pi_{k-1}$ reduces to
		\begin{equation*}
			\Pi_{k-1} = \operatorname{aff} \left\{ \cX^{j_{k-1,\tau}}, \ldots, \cX^{k-1}, \mathcal{P}_{\pi_{k-1}}(\cX^{k-1}) \right\} = \operatorname{aff} \left\{ \cX^{j_{k-1,\tau}}, \ldots, \cX^{k-1} \right\}.
		\end{equation*}
		Since $\cX^i \in \Pi_{k-2}$ for all $i = j_{k-1,\tau}, \ldots, k-1$, the search subspace satisfies $\Pi_{k-1} \subseteq \Pi_{k-2}$.
		Given that $\cX^{k-1}$ is the minimizer of $\lVert \cX - \cX_*^0 \rVert_F^2$ over $\Pi_{k-2}$ and lies in $\Pi_{k-1} \subseteq \Pi_{k-2}$, it must also be the minimizer over $\Pi_{k-1}$, i.e. $\cX^{k-1} = \arg\min_{\cX \in \Pi_{k-1}} \| \cX - \cX^0_{*} \|_F^2.$	
		Therefore, we deduce that $\cX^k = \cX^{k-1}$. This implies $\cA \ast \cX^k = \cB$ provided $\cA \ast \cX^{k-1} = \cB$, which contradicts the assumption that $\cA \ast \cX^k \neq \cB$.
		We conclude that if $\cA \ast \cX^k \neq \cB$, then $\cA \ast \cX^{k-1} \neq \cB$.
		This completes the proof the lemma.
	\end{proof}
	
	\begin{proof}[Proof of Lemma~\ref{ensure-comput}]
		According to the definition of $\gamma_k$ in \eqref{def_gammak}, we can rewrite $\gamma_k$ as
		\begin{equation*}
			\begin{aligned}
				\gamma_k&=\left\langle P_{\pi_{k}}\left(\cX^{k}\right)-\cX^{k}, \cX_{*}^{0} - \cX^{k} \right\rangle\\
				&= 
				\frac{1}{2}\left(\|\cX^k - \cX^*\|_{F}^2 
				-
				\|\cP_{\pi_{k}}(\cX^k) - \cX^*\|_{F}^2 
				+ 
				\| \cX^k - \cP_{\pi_{k}}(\cX^k)\|_{F}^2\right)\\
				&=\frac{1}{2}\left(\|r_{\pi_k}(\cX^k)\|_F^2
				+ 
				\| \cX^k - \cP_{\pi_{k}}(\cX^k)\|_{F}^2\right),
			\end{aligned}
		\end{equation*}
		where the last equation follows from the identity \eqref{successive_Pythagoras}.
		If $\cA \ast \cX^k \neq \cB$, then by Lemma~\ref{equiv} we have
		$r_{\pi_k}(\cX^k) \neq 0$ and $\cX^k - \cP_{\pi_{k}}(\cX^k)\neq0$, and hence $\gamma_k > 0$.
		
		Next, we prove that if $\cA\ast\cX^k\neq\cB$, $M_k$ has full column rank for all $k\geq0$.
		For $k=0$, $M_0=
			\vec{d}^0$ with $\vec{d}^0=\tv(\cP_{\pi_0}(\cX^0)-\cX^0)$.
		Thus, $M_0$ is of full column rank provided that $\cP_{\pi_0}(\cX^0) - \cX^0 \neq 0$.
		From Lemma \ref{equiv}, we know $\cP_{\pi_0}(\cX^0) - \cX^0 \neq 0$ if and only if $\cA\ast \cX^0\neq\cB$.
		Hence, if $\cA\ast \cX^0\neq\cB$, $M_0$ is of full column rank.
		
		By induction, assume that if $\cA\ast\cX^{k}\neq\cB$, $M_{k}$ has full column rank.
		Suppose that $\cA\ast\cX^{k+1}\neq\cB$, we obtain $\cA\ast\cX^{k}\neq\cB$ by Lemma \ref{equiv}, which ensuring that $M_{k}$ has full column rank.
		The tube vectorization of iterate $\cX^{k+1}$ admits the following formula
		\[
		\vec{x}^{k+1} = \vec{x}^k + M_ks^k 
		= \vec{x}^k+\sum_{i=1}^{k-j_{k,\tau}}\, s_i^k \left(\vec{x}^{j_{k,\tau}+i-1}-\vec{x}^k\right) + s_{k-j_{k,\tau}+1}^k  \vec{d}^k.
		\]
		Under the assumption that $\cA\ast\cX^k\neq\cB$, we have $\gamma_k>0$, and hence, by the expression of $s_{k-j_{k,\tau}+1}^k$ in \eqref{vec_update}, $s_{k-j_{k,\tau}+1}^k\neq0$.
		Since $s_{k-j_{k,\tau}+1}^k\neq0$ and $M_k$ has full column rank, it holds that $\vec x^{j_{k,\tau}}, \ldots, \vec{x}^{k+1}$ are affinely independent,
		which implies that
		$\vec x^{j_{k,\tau}}-\vec{x}^{k+1}, \ldots, \vec{x}^{k}-\vec{x}^{k+1}$ are linearly independent.
		Suppose, for contradiction, that $M_{k+1}$ does not have full column rank, i.e. $\vec x^{j_{k+1,\tau}}-\vec{x}^{k+1}, \ldots, \vec{x}^{k}-\vec{x}^{k+1}, \vec{d}^{k+1}$ are linearly dependent.
		Then, there exist scalars $\{c_i\}_{i=j_{k+1,\tau}}^k$, not all zero, satisfying 
		\(
		\vec{d}^{k+1}
		=
		\sum_{i=j_{k+1,\tau}}^{k} c_i\, (\vec x^i - \vec x^{k+1})
		\). 
		Taking the inner product with $\vec x_*^0 - \vec x^{k+1}$ yields
		\[
		\gamma_{k+1}=\langle \vec x_*^0 - \vec x^{k+1}, \vec{d}^{k+1} \rangle
		=
		\sum_{i=j_{k+1,\tau}}^{k} c_i\,
		\langle \vec x_*^0 - \vec x^{k+1}, \vec x^i - \vec x^{k+1} \rangle
		= 0,
		\]
		where the last equality follows from $\langle \vec x_*^0 - \vec x^{k+1}, \vec x^i - \vec x^{k+1} \rangle
		= 0$ for $i=j_{k+1,\tau},\ldots,k$.
		Since $\cA\ast\cX^{k+1}\neq\cB$, we obtain $\gamma_{k+1}>0$, which leads to a contradiction.
		Therefore, for all $k\geq0$, $M_{k}$ has full column rank if $\cA\ast\cX^{k}\neq\cB$ . This completes the proof of the lemma.
	\end{proof}
	
	\subsection{Proof of Theorem~\ref{TKGK_convergence_thm}}
	The following lemma is essential for proving  Theorem~\ref{TKGK_convergence_thm}.
	\begin{lemma}\label{gamma_square}
		For any $k \ge 0$, we have the following identity
		\[
		\langle \cX^k - \cX_{*}^{0}, \cX^k - \cP_{\pi_{k}}(\cX^k) \rangle^{2}
		= 
		\|\cX^k - \cP_{\pi_{k}}(\cX^k)\|_{F}^{2}\|r_{\pi_{k}}(\cX^k)\|_{F}^{2}
		+ \langle \cP_{\pi_{k}}(\cX^k) - \cX_{*}^{0}, \cX^k - \cP_{\pi_{k}}(\cX^k) \rangle^{2}.
		\]
	\end{lemma}
	%
	\begin{proof}
		Recall that $\cP_{\pi_k}(\cdot)$ and $r_{\pi_k}(\cdot)$ are defined in~\eqref{full-projection-oper} and~\eqref{def_r_k}, respectively.
		Given that $\langle \cX^k - \cX_*^0,\, \cX^k - \cP_{\pi_k}(\cX^k) \rangle
		=\|\cX^k - \cP_{\pi_k}(\cX^k)\|_F^2+\langle \cP_{\pi_k}(\cX^k) - \cX_*^0,\,\cX^k - \cP_{\pi_k}(\cX^k) \rangle $, it follows that
		\begin{equation}\label{lemma7_1_proof}
			\begin{aligned}
				\langle \cX^k - \cX_*^0,\, \cX^k - \cP_{\pi_k}(\cX^k) \rangle^2
				& =
				\|\cX^k - \cP_{\pi_k}(\cX^k)\|_F^4
				+ \langle \cP_{\pi_k}(\cX^k) - \cX_*^0, \cX^k - \cP_{\pi_k}(\cX^k) \rangle^2 \\
				& \quad + 2 \|\cX^k - \cP_{\pi_k}(\cX^k)\|_F^2
				\langle \cP_{\pi_k}(\cX^k) - \cX_*^0, \cX^k - \cP_{\pi_k}(\cX^k) \rangle. 
			\end{aligned}
		\end{equation}
		From the definition of $\gamma_k$ in \eqref{def_gammak} and its equivalent form in \eqref{gammak}, we have
		$$ \langle \cX^k - \cX_*^0, \cX^k - \cP_{\pi_k}(\cX^k) \rangle=\gamma_k=\frac{1}{2}\|r_{\pi_k}(\cX^k)\|_F^2
		+
		\frac{1}{2}\|\cX^k - \cP_{\pi_k}(\cX^k)\|_F^2.$$
		Hence, we can rewrite $\langle \cP_{\pi_k}(\cX^k) - \cX_*^0,
		\cX^k - \cP_{\pi_k}(\cX^k) \rangle$ as
		\begin{equation}\label{sun_0404_2}
			\begin{aligned}
				\langle \cP_{\pi_k}(\cX^k) - \cX_*^0,
				\cX^k - \cP_{\pi_k}(\cX^k) \rangle
				& = - \|\cX^k - \cP_{\pi_k}(\cX^k)\|_F^2 + \langle \cX^k - \cX_*^0, \cX^k - \cP_{\pi_k}(\cX^k) \rangle\\
				&=\frac{1}{2}\|r_{\pi_k}(\cX^k)\|_F^2-\frac{1}{2}\|\cX^k - \cP_{\pi_k}(\cX^k)\|_F^2.
			\end{aligned}
		\end{equation}
		Substituting it into \eqref{lemma7_1_proof}, we obtain
		\[
		\langle \cX^k - \cX_*^0,\, \cX^k - \cP_{\pi_k}(\cX^k) \rangle^2 
		=
		\|\cX^k - \cP_{\pi_k}(\cX^k)\|_F^2 \|r_{\pi_k}(\cX^k)\|_F^2
		+ \langle \cP_{\pi_k}(\cX^k) - \cX_*^0,\,
		\cX^k - \cP_{\pi_k}(\cX^k) \rangle^2,
		\]
		which proves the result.
	\end{proof}
	
	\begin{proof}[Proof of Theorem~\ref{TKGK_convergence_thm}]
		We first demonstrate that $\cA\ast\cX^k=\cB$ implies that $\cX^k=\cX_*^0$.
		Suppose that $\cA\ast\cX^k=\cB$ and $\cA\ast\cX^i\neq\cB$ for $i=0,1,\ldots,k-1$. 
		From the proof of Theorem \ref{trrk_convergence_thm}, we know that $\cX^k\in\cX^0 + \operatorname{range}_K(\cA^\dagger)$, which ensures that
		$
		(I-\cA^\dagger\ast \cA)\ast\cX^{k}=(I-\cA^\dagger\ast \cA)\ast\cX^0.
		$
		Using the assumption $\cA\ast\cX^k=\cB$, we obtain
		$$\cX^k=\cA^\dagger\ast\cA\ast\cX^k+(I-\cA^\dagger\ast \cA)\ast\cX^0=\cA^\dagger\ast\cB+(I-\cA^\dagger\ast \cA)\ast\cX^0=\cX_*^0.$$	
		We next consider the case where $\cA\ast\cX^k\neq\cB$.
		Recall the iteration scheme of $\cX^{k+1}$ in Algorithm~\ref{TKGK}, which admits the vectorized form
		\(
		\vec x^{k+1}=\vec x^{k}+M_k s^k.
		\)
		Thus, we obtain
		\begin{align*}
			\|\cX^{k+1} - \cX_*^0\|_F^2
			&=\lVert \vec{x}^{k+1}-\vec{x}_*^0 \rVert_2^2\\
			& = \|\vec{x}^{k} - \vec{x}_*^0\|_2^2 
			+ (s^k)^\top M_k^\top M_k s^k
			+ 2(s^k)^\top M_k^\top(\vec x^{k}-\vec x_*^0) \\
			&=\|\vec{x}^{k} - \vec{x}_*^0\|_2^2 
			-(s^k)^\top M_k^\top M_k s^k\\
			&=\|\vec{x}^{k} - \vec{x}_*^0\|_2^2 
			- \gamma_k s^{k}_{k-j_{k,\tau}+1}\\
			& = \|\cX^{k} - \cX_*^0\|_F^2 - \gamma_k s^{k}_{k-j_{k,\tau}+1},
		\end{align*}
		where the third and fourth equalities follow from \eqref{optcond2} and \eqref{eq-s_k}, respectively.
		Using the equivalent form of $s^{k}_{k-j_{k,\tau}+1}$ in \eqref{oversk_undersk}, we can further express it as
		\begin{align*}
			s^{k}_{k-j_{k,\tau}+1}
			=  \frac{\gamma_k}{\lVert (I-V_kV_k^\dagger)\vec{d}^k\rVert_2^2}
			=\frac{\gamma_k}{\|\vec d^k\|_2^2}
			\left(1-\frac{\|V_k V_k^\dagger\vec d^k\|_2^2}{\|\vec d^k\|_2^2}\right)^{-1},
		\end{align*}
		where the last equality follows from 
		$\lVert (I-V_kV_k^\dagger)\vec{d}^k\rVert_2^2=\lVert\vec{d}^k\rVert_2^2- \lVert V_kV_k^\dagger\rVert_2^2\|\vec{d}^k\|^2_2$. 
		Recall the definitions of \(\gamma_k\) in~\eqref{def_gammak} and $\zeta_k$ in \eqref{rate-betak}, 
		it holds that
		\(
		\frac{\gamma_k^2}{\|\vec d^k\|_2^2} = \zeta_k^2\|\cX^{k}-\cX_*^0\|_F^2.
		\)
		As 
		\(
		\beta_k= \left(1 - \frac{\|V_k V_k^\dagger\vec d^k\|_2^2}{\|\vec d^k\|_2^2}\right)^{-1}
		\)  in~\eqref{rate-betak},
		we conclude that $s^{k}_{k-j_{k,\tau}+1}=\frac{\beta_k\zeta_k^2\|\cX^{k}-\cX_*^0\|_F^2}{\gamma_k}$ and
		\begin{align*}
			\|\cX^{k+1} - \cX_*^0\|_F^2 \leq (1-\beta_k\zeta_k^2)\|\cX^{k} - \cX_*^0\|_F^2.
		\end{align*}
		
		Next, we prove that $1 - \beta_k\zeta_{k}^2 \le \|\bc\left(T_{\pi_{k}} \ast \cA^{\dagger} \ast \cA\right) \|_{2}^2$.
		Since \(M_k\) has full column rank, it follows that \(\vec d^k\notin\mathrm{Range}(V_k)\).
		Consequently, because \(V_k V_k^\dagger\) is the orthogonal projector onto \(\mathrm{Range}(V_k)\),
		it holds that 
		\(
		0 < 1 - \frac{\|V_k V_k^\dagger\vec d^k\|_2^2}{\|\vec d^k\|_2^2} \le 1,
		\)
		which implies that $\beta_k\geq1$.
		Therefore, to establish $1 - \beta_k\zeta_{k}^2 \le \|\bc\left(T_{\pi_{k}} \ast \cA^{\dagger} \ast \cA\right) \|_{2}^2$, it suffices to prove the inequality $1 - \zeta_{k}^2 \le \|\bc\left(T_{\pi_{k}} \ast \cA^{\dagger} \ast \cA\right) \|_{2}^2$. 
		From the definition of $\zeta_k$ in \eqref{rate-betak} and  Lemma~\ref{gamma_square}, $1-\zeta_k^2$ is expressed as 
		\begin{align}\label{decomposed_zeta_k}
			1 - \zeta_{k}^2 
			& = 
			1 - \frac{\left\|r_{\pi_k}(\cX^k)\right\|_F^2}{\left\|\cX^k - \cX_*^0\right\|_F^2} 
			- \frac{\langle\cP_{\pi_k}(\cX^k) - \cX_*^0, \cX^k - \cP_{\pi_k}(\cX^k)\rangle^2}
			{\left\|\cX^k - \cX_*^0\right\|_F^2\left\|\cX^k - \cP_{\pi_k}(\cX^k)\right\|_F^2}.
		\end{align}
		Using the equation in \eqref{successive_Pythagoras},
		We further obtain 
		\[
		1 - \zeta_k^2 
		\le
		1 - \frac{\left\|r_{\pi_k}(\cX^k)\right\|_F^2}{\left\|\cX^k - \cX_*^0\right\|_F^2}
		= 
		\frac{\left\|\cP_{\pi_k}(\cX^k) - \cX_*^0\right\|_F^2}{\left\|\cX^k - \cX_*^0\right\|_F^2} .
		\] 
		From the inequality \eqref{sun_0404} in the proof of Theorem~\ref{trrk_convergence_thm}, it holds that
		$$\left\|\cP_{\pi_k}(\cX^k) - \cX_*^0\right\|_F^2\le
		\left\|\bc(\cT_{\pi_k} \ast \cA^{\dagger} \ast \cA)\right\|_2^2\left\|\cX^k - \cX_*^0\right\|_F^2,$$ 
		which results in $1 - \zeta_k^2 
		\le\left\|\bc(\cT_{\pi_k} \ast \cA^{\dagger} \ast \cA)\right\|_2^2.$
		
		If $\|r_{\pi_k}(\cX^k)\|_F^2
		\neq
		\|\cX^k - \cP_{\pi_k}(\cX^k)\|_F^2$, then the last term of~\eqref{decomposed_zeta_k} gives
		\[
		\frac{\langle\cP_{\pi_k}(\cX^k) - \cX_*^0, \cX^k - \cP_{\pi_k}(\cX^k)\rangle^2}
		{\left\|\cX^k - \cX_*^0\right\|_F^2\left\|\cX^k - \cP_{\pi_k}(\cX^k)\right\|_F^2}
		=
		\frac{\left(\|r_{\pi_k}(\cX^k)\|_F^2
			-
			\|\cX^k - \cP_{\pi_k}(\cX^k)\|_F^2\right)^2}{4\left\|\cX^k - \cX_*^0\right\|_F^2\left\|\cX^k - \cP_{\pi_k}(\cX^k)\right\|_F^2} >0, 
		\]
		where the equality follows from~\eqref{sun_0404_2}.
		Hence we complete the proof. 
	\end{proof}

	\subsection{Proof of Proposition~\ref{basis_update_lemma} and Theorem~\ref{thm_equiv_Algo2_Algo3}} \label{appendix_sec_4}

	\begin{proof}[Proof of Proposition~\ref{basis_update_lemma}]
		We prove by induction on $k \ge 1$ that the sequence $\{\cU_i\}_{i=j_{k,\tau}}^k$ is orthogonal, 
		i.e., $\langle \cU_i, \cU_j \rangle = 0$ for all $i \ne j$ with $j_{k,\tau} \le i, j \le k$.
		For $k=1$, we have
		\[
		\langle \cU_0, \cU_1 \rangle
		=
		\langle \cU_0, \cD_1 \rangle
		-
		\frac{\langle \cU_0,\cD_1\rangle}{\|\cU_0\|_F^2}
		\|\cU_0\|_F^2
		= 0.
		\]
		Assume that for some $k \geq 1$, $\langle \mathcal{U}_i, \mathcal{U}_j \rangle = 0$, for all $i \neq j$ with $j_{k,\tau} \leq i,j \leq k$.
		Since $j_{k,\tau}\leq j_{k+1,\tau}$, 
		the induction hypothesis implies that $\langle\cU_i, \cU_j\rangle=0$ for all $j_{k+1,\tau} \le i \neq j \le k$. 
		It remains to show that 
		\(
		\langle \mathcal{U}_{k+1}, \mathcal{U}_j \rangle = 0 \) for all  $j$ satisfying $j_{k+1,\tau} \leq j \leq k.$ 
		From the definition of $\cU_{k+1}$ in \eqref{xie-eq-0323},
		for any $j$ such that $j_{k+1,\tau} \le j \le k$, we get
		\begin{align*}
			\langle \cU_j, \cU_{k+1} \rangle
			&=
			\langle \cU_j, \cD_{k+1} \rangle
			-
			\sum_{i=j_{k+1,\tau}}^{k}
			\frac{\langle \cU_i,\cD_{k+1}\rangle}{\|\cU_i\|_F^2}
			\langle \cU_j, \cU_i \rangle.
		\end{align*} 
		By the induction hypothesis, $\langle \cU_j, \cU_i \rangle = 0$ for all $i \neq j$, hence we have
		\[
		\langle \cU_j, \cU_{k+1} \rangle
		=
		\langle \cU_j, \cD_{k+1} \rangle
		-
		\frac{\langle \cU_j,\cD_{k+1}\rangle}{\|\cU_j\|_F^2}
		\|\cU_j\|_F^2
		= 0.
		\]
		This completes the proof of the orthogonality property.
		
		Next, we  prove that $\operatorname{span}\{\mathcal{U}_{j_{k,\tau}}, \dots, \mathcal{U}_{k-1}\} = \operatorname{span}\{\mathcal{X}^{j_{k,\tau}} - \mathcal{X}^{k}, \dots, \mathcal{X}^{k-1} - \mathcal{X}^{k}\}$. 
		Consider the equivalent expression of the subspace 
		\begin{align*}
			\operatorname{span}\{\mathcal{U}_{j_{k,\tau}}, \dots, \mathcal{U}_{k-1}\} = \operatorname{span}\{\mathcal{X}^{j_{k,\tau}+1} - \mathcal{X}^{j_{k,\tau}}, \dots, \mathcal{X}^{k} - \mathcal{X}^{k-1}\},
		\end{align*}
		where this equality follows directly from the relation $\mathcal{X}^{i+1} - \mathcal{X}^i = \lambda_i \mathcal{U}_i$ with $\lambda_i \neq 0$.
		It therefore suffices to show that
		$
		\operatorname{span}\{\mathcal{X}^{j_{k,\tau}+1} - \mathcal{X}^{j_{k,\tau}}, \dots, \mathcal{X}^{k} - \mathcal{X}^{k-1}\}
		=\{
		\operatorname{span}\mathcal{X}^{j_{k,\tau}} - \mathcal{X}^{k}, \dots, \mathcal{X}^{k-1} - \mathcal{X}^{k}\}.
		$
		Note that for any $j_{k,\tau} \le i \le k-1$, 
		$$
		\cX^{i+1} - \cX^{i} = \cX^{i+1} - \cX^k - (\cX^{i} - \cX^k) \in \spn\{\cX^{j_{k,\tau}}-\cX^k,\ldots,\cX^{k-1}-\cX^k\},
		$$
		hence $\operatorname{span}\{\mathcal{X}^{j_{k,\tau}+1} - \mathcal{X}^{j_{k,\tau}}, \dots, \mathcal{X}^{k} - \mathcal{X}^{k-1}\}\subseteq\spn\{\cX^{j_{k,\tau}}-\cX^k,\ldots,\cX^{k-1}-\cX^k\}.$ 
		Conversely, for any $j_{k,\tau} \le i \le k-1$, 
		$$
		\cX^i - \cX^k = - (\cX^{i+1} - \cX^i + \cdots + \cX^k - \cX^{k-1}) 
		\in \spn\{\cX^{j_{k, t}+1} - \cX^{j_{k,\tau}}, \ldots, \cX^{k} - \cX^{k-1}\},
		$$
		and therefore $\spn\{\cX^{j_{k,\tau}}-\cX^k,\ldots,\cX^{k-1}-\cX^k\}\subseteq\spn\{\cX^{j_{k, t}+1} - \cX^{j_{k,\tau}}, \ldots, \cX^{k} - \cX^{k-1}\}.$ 
		Thus we prove the two linear spans coincide. 
	\end{proof}
	
	\begin{proof}[Proof of Theorem~\ref{thm_equiv_Algo2_Algo3}]
		It suffices to show that the vectorized iterates coincide.
		We prove by induction, and the case $k=0$ is immediate.
		
		Assume that $\vec{x}^i$ coincide for all $i \le k$. Then the subspace
		\(
		\operatorname{Range}(V_k) = \spn\{\vec{x}^{j_{k,\tau}}-\vec{x}^k, \ldots, \vec{x}^{k-1}-\vec{x}^k\}
		\)
		is identical for both algorithms.
		Since $\{\pi_k\}_{k \ge 0}$ is identical, both algorithms compute the same $\vec{d}^k$ and $\gamma_k$.
		By Proposition~\ref{basis_update_lemma}, Algorithm~\ref{G-S-based TKGK} constructs 
		an orthogonal basis $\{\vec u_i\}_{i=j_{k,\tau}}^{k-1}$ of $\operatorname{Range}(V_k)$.  
		Let $U_k = [\vec u_{j_{k,\tau}},\ldots,\vec u_{k-1}]$, then 
		\(
		\operatorname{Range}(U_k) = \operatorname{Range}(V_k),
		\)
		and hence the corresponding orthogonal projectors coincide, i.e.,
		\[
		U_k(U_k^\top U_k)^{-1}U_k^\top = V_k V_k^\dagger.
		\]
		Thus the Step~8 in Algorithm~\ref{G-S-based TKGK} produces the update direction
		\[
		\vec{u}_k = \vec d^k - \sum_{i=j_{k,\tau}}^{k-1}\frac{\langle \vec u_i,\vec d^k\rangle}{\|\vec u_i\|_2^2}\vec u_i
		= \left( I - U_k(U_k^{\top}U_k)^{-1}U_k^{\top}\right) \vec d^k = (I - V_k V_k^\dagger)\vec d^k.
		\]
		From \eqref{vec_update}, Algorithm~\ref{TKGK} updates along the direction
		\(
		(I - V_k V_k^\dagger)\vec d^k
		\)
		with stepsize $\underline{s^k} = \frac{\gamma_k}{\langle\vec{d}^k, (I - V_kV_k^{\dagger})\vec d^k\rangle}$.
		Therefore, the update directions coincide.
		Moreover, from \eqref{idempotent_I-V_kV_kdagger},
		$\underline{s^k}$ coincides exactly with the stepsize
		$\lambda_k = \frac{\gamma_k}{\|\vec{u}_k\|_2^2}$ computed as Step~4 in Algorithm~\ref{G-S-based TKGK}.
		Therefore, both algorithms generate the same $\vec{x}^{k+1}$.
		By induction, the sequences $\{\vec{x}^k\}_{k \ge 0}$ coincide. 
	\end{proof}

\end{document}